\documentclass[12pt]{amsart}

\allowdisplaybreaks[1]

\usepackage[all]{xy}
\usepackage{longtable}
\usepackage{array}
\usepackage{tikz}
\usetikzlibrary{arrows}
\usepackage{amsxtra,amsmath,amsfonts,amscd, amssymb, mathrsfs}
\usepackage{tikz-cd}
\usepackage{mathtools}
\usepackage{tabu}
\usepackage{subcaption}
\usepackage{float}
\usepackage{color}
\usepackage[T1]{fontenc}
\usepackage{float}
\usepackage{thmtools}
\usepackage{enumitem}
\usepackage{thm-restate}
\usepackage[colorlinks]{hyperref}
\usepackage{mathdots}
\usepackage{adjustbox}
\usepackage{tabularx}

\colorlet{linkcolor}{green!50!black}
\hypersetup{allcolors=linkcolor}

\DeclareTextFontCommand{\emph}{\boldmath\bfseries}

\numberwithin{equation}{section} % equation own row

\usepackage[ left = 1.2in, right = 1.2in,
             top = 1.2in, bottom = 1.2in ]{geometry}

\usepackage{charter} % roman text, more bold than usual

\DeclarePairedDelimiter{\ceil}{\lceil}{\rceil}

\DeclareMathOperator{\Span}{span}

\DeclareMathOperator{\rank}{rank}
\DeclareMathOperator{\res}{res}

\DeclareMathOperator{\Hom}{Hom}

\DeclareMathOperator{\Gr}{Gr}

\DeclareMathOperator{\link}{link}

\newcommand{\cA}{\mathcal{A}}

\newcommand{\cC}{\mathcal{C}}

\newcommand{\cE}{\mathcal{E}}
\newcommand{\cF}{\mathcal{F}}
\newcommand{\cG}{\mathcal{G}}

\newcommand{\cL}{\mathcal{L}}

\newcommand{\cW}{\mathcal{W}}

\theoremstyle{plain} % title bold and text italics
\newtheorem{thm}{Theorem}[section]
\newtheorem{cor}[thm]{Corollary}

\newtheorem{lem}[thm]{Lemma}
\newtheorem{prop}[thm]{Proposition}

\theoremstyle{definition} % title bold and text roman
\newtheorem{example}[thm]{Example}
\newtheorem{defn}[thm]{Definition}

\newtheorem{notation}[thm]{Notation}

\theoremstyle{remark} % title italics and text roman
\newtheorem{remark}[thm]{Remark}

\newcommand{\R}{\mathbb{R}}

\newcommand{\C}{\mathbb{C}}
\newcommand{\Z}{\mathbb{Z}}

\newcommand{\ov}[1]{\overline{#1}}

\newcommand{\De}{{\Delta}}
\newcommand{\lk}{\mathrm{link}}
\newcommand{\Hilb}{\mathrm{Hilb}}

\newcommand{\tdeg}{\mathrm{tdeg}}
\newcommand{\tdim}{\mathrm{tdim}}
\newcommand{\mmax}{\boldsymbol{m}}
\DeclareMathOperator{\BBox}{{Box}}

\DeclareMathOperator{\gen}{{gen}}

\DeclareMathOperator{\id}{id}
\DeclareMathOperator{\Ehr}{Ehr}

\title[Filtrations on combinatorial intersection cohomology]{Filtrations on combinatorial intersection cohomology and invariants of subdivisions}

\author{Ling Hei Tsang}
\address{Ling Hei Tsang, Department of Mathematics, The Ohio State University University, 231 W. 18th Ave., Columbus, OH 43210}
\email{tsang.79@osu.edu}

\keywords{Fan subdivisions, toric $h$-polynomial, $h^*$-polynomial, $cd$-index, filtrations}

\subjclass[2010]{52B05, 14F43, 52B20}

\begin{document}

\begin{abstract}
Motivated by definitions in mixed Hodge theory,
we define the weight filtration and the monodromy weight filtration on the combinatorial intersection cohomology of a fan.
These filtrations give a natural definition of the multivariable invariants of subdivisions of polytopes, lattice polytopes and fans, namely the mixed $h$-polynomial, the refined limit mixed $h^*$-polynomial,
and the mixed $cd$-index,
defined by Katz--Stapledon and Dornian--Katz--Tsang.
Previously,
only the refined limit mixed $h^*$-polynomial had a geometric interpretation, which came from 
filtrations on the cohomology of a sch\"on hypersurface.
Consequently, we generalize a positivity result on the mixed $h$-polynomial by Katz and Stapledon using the relative hard Lefschetz theorem of Karu.
\end{abstract}

\maketitle

\section{Introduction}\label{sec:intro}

There are enumerative polynomial invariants of polytopes and lattice polytopes,
coming from counting faces and lattice points.
These invariants are motivated by the cohomology of toric varieties and hypersurfaces in algebraic tori.
They admit multivariable enrichments, that are invariants of subdivisions,
defined using a formalism inspired by Kazhdan--Lusztig theory.
A natural question is if one can obtain these invariants of subdivisions directly,
using the theory of \textit{combinatorial intersection cohomology}
though the theory does not obviously come from geometry.
The objective of this paper is to introduce filtrations analogous to the \textit{weight and Hodge filtrations} in \textit{mixed Hodge theory} to give a natural definition of these invariants of subdivisions.

The combinatorial invariants are the following:
\vspace{-3pt}
\begin{itemize}[leftmargin=24pt]
\item The \emph{toric $h$-polynomial} $h(P; t)$ is a combinatorial invariant of polytopes and polyhedral complexes defined by Stanley \cite[Section~2]{Sta87}.
It is motivated by the intersection cohomology of toric varieties.
When a polyhedral complex $\mathcal{S}$ is simplicial, its toric $h$-polynomial encodes its face numbers.

\item The \emph{$h^*$-polynomial} $h^*(P; t)$ is a combinatorial invariant in Ehrhart theory,
coming from counting lattice points in lattice polytopes and lattice polyhedral complexes.
Notably, the $h^*$-polynomials of lattice polytopes admit a geometric interpretation,
using the cohomology of hypersurfaces in algebraic tori \cite[Section~4]{DK86}.

\item The \emph{$cd$-index} $\Phi_\Delta(c, d) \in \R\langle c, d \rangle$ is a polynomial in non-commuting variables $c$ and $d$ that encodes the flag $f$-numbers of a fan,
defined by Fine and Bayer--Klapper \cite[Theorem~4]{BK91}.
\end{itemize}

\begin{table}[H]
\footnotesize
\begin{tabular}{@{} >{\raggedright\arraybackslash\hspace{0pt}}p{7em} | >{\raggedright\arraybackslash\hspace{0pt}}p{11.3em} |
>{\raggedright\arraybackslash\hspace{0pt}}p{11.3em} |
>{\raggedright\arraybackslash\hspace{0pt}}p{11.3em} @{}}
Invariants & Toric $h$-polynomial & $h^*$-polynomial & $cd$-index \\
\hline
\hline
Local invariants & Local $h$-polynomial $\ell^h_P(\mathcal{S}; t)$ \cite{Sta92} & Local $h^*$-polynomial $\ell^*(P; t)$ \cite{Sta92}
and local limit mixed $h^*$-polynomial $\ell^*(P, \mathcal{S}; u, v)$ \cite{KS16a}
& Local $cd$-index $\ell^\Phi_\Delta(c, d)$ \cite{Kar06} \\
\hline
Decomposition theorems & See \eqref{eq:decompforh}
\cite{Ath12, KS16a, Sta92}
& See \eqref{eq:hstardecomp} \cite{KS16a, Sta92}
& See \eqref{eq:cddecomp}
\cite{EK07, DKT20}
 \\
\hline
Multivariable invariants of subdivisions  &  Mixed $h$-polynomial $h_P(\mathcal{S}; u, v)$  \cite{KS16a} &
Refined limit mixed $h^*$-polynomial $h^*(P, \mathcal{S}; u, v, w)$ \cite{KS16a}
&
Mixed $cd$-index $\Omega_\pi(c', d', c, d)$ \cite{DKT20}
\\
\hline
Corresponding pure sheaves & Minimal extension sheaf $\cL_\Delta$ with structure sheaf $\cA$ \cite{BBFK02, BL03, Kar04}
& Ehrhart sheaf $\cE_\Delta$ \cite{Kar08} & Minimal extension sheaf $\cL_\Delta$ with structure sheaf $\cC$ \cite{Kar06}
\end{tabular}
\end{table}

There is a procedure for extending these invariants to multivariable invariants of subdivisions,
namely the mixed $h$-polynomial, the refined limit mixed $h^*$-polynomial, and the mixed $cd$-index,
using their local invariants and decomposition theorems:
\begin{align*}
h_P(\mathcal{S}; u, v) & = \sum_{Q\subseteq P}
v^{\dim (Q)+1} \ell^h_{Q}(\mathcal{S}|_Q; u v^{-1}) \cdot h(\lk_P Q; uv);
\\
h^*(P, \mathcal{S}; u, v, w) & = \sum_{Q\subseteq P} w^{\dim (Q) + 1} \ell^*(Q, \mathcal{S}|_Q; u, v) \cdot h(\lk_P Q; uvw^2);
\\
\Omega_\pi(c', d', c, d) & = \sum_{\sigma\in \Delta} \ell_{\pi^{-1}(\langle \sigma \rangle)}^\Phi(c', d') \otimes \Phi_{\lk_\Delta\sigma}(c, d).
\end{align*}
The first two of which, namely the \emph{mixed $h$-polynomial} $h_P(\mathcal{S}; u, v)$ and the \emph{refined limit mixed $h^*$-polynomial} $h^*(P, \mathcal{S}; u, v, w)$,
were defined by Katz and Stapledon \cite[Section~5, 9]{KS16a} for a lattice polyhedral subdivision $\mathcal{S}$ of a lattice  polytope $P$, using a formalism motivated by Kazhdan--Lusztig theory.
Later, Dornian, Katz and the author \cite[Section~6]{DKT20} defined the \emph{mixed $cd$-index} $\Omega_\pi(c', d', c, d)$ in a similar fashion for a fan subdivision $\pi\colon \Sigma\to \Delta$.

Among these invariants of subdivisions,
the refined limit mixed $h^*$-polynomial has a geometric interpretation \cite[Section~5]{KS16b}.
We give a brief description here; see \cite{PS08} for background in mixed Hodge theory.
If $\mathcal{S}$ is a \textit{regular} lattice polyhedral subdivision of a lattice polytope $P$,
then one can pick a \textit{sch\"on} hypersurface $X^\circ \subseteq (\C(t)^*)^{\dim(P)}$ with associated Newton polytope and polyhedral subdivision $(P, \mathcal{S})$.
Here sch\"on is a smoothness condition; see for example \cite[Definition~6.4.18]{MS15}.
The hypersurface $X^\circ$ corresponds to a locally trivial fibration $f\colon X^\circ \to \mathbb{D}^*$, where $\mathbb{D}^*$ is a small punctured disk about the origin in $(\C(t)^*)^{\dim(P)}$.
The complex cohomology with compact support $H^m_c(X^\circ_{\gen})$ of
a nonzero fiber $X^\circ_{\gen} := f^{-1}(t)$
is equipped with a decreasing filtration $F^\bullet$ called the \textit{Hodge filtration}, and two increasing filtrations $W_\bullet$ and $M_\bullet$ called the \textit{weight filtration} and the \textit{monodromy weight filtration}, respectively.
Then the \textit{refined limit Hodge--Deligne polynomial} of $X^\circ$,
a polynomial defined using filtrations on $H^m_c(X^\circ_{\gen})$,
gives the refined limit mixed $h^*$-polynomial of $(P, \mathcal{S})$ \cite[Theorem~1.5]{KS16b}.

This leads to the following question:
\textit{Can we obtain these combinatorial invariants of subdivisions intrinsically and systematically?}
In this paper, we give an affirmative answer to the question by defining filtrations on the combinatorial intersection cohomology of a fan.
In particular, we study the invariants of a fan instead of a polytope.

The theory of combinatorial intersection cohomology of a fan was introduced independently by
Barthel--Brasselet--Fieseler--Kaup \cite{BBFK02}
and
Bressler--Lunts \cite{BL03}.
The idea is to consider a \emph{pure sheaf} $\cF$
on a fan $\Delta$, such that the Poincar\'e polynomial of the graded vector space $\ov{\cF(\Delta)}$ (arising as a reduction of the module of
global sections) coincides with the desired combinatorial invariant.
Here, a pure sheaf is a locally free flabby sheaf of $\Z_{\geq 0}$-graded $\cA$-modules; 
see Definition~\ref{defn:puresheaf}.
Notably, the minimal extension sheaf $\cL_\Delta$ with structure sheaf $\cA$
is the pure sheaf that gives the toric $h$-polynomial;
see Section~\ref{sec:hpoly}.
Hence the theory serves as a substitute for intersection cohomology when the fan is non-rational.
Indeed, this is crucial in Karu's proof of the hard Lefschetz theorem for non-rational polytopes \cite[Theorem~0.1]{Kar04}.
Similarly, there are pure sheaves that give the $h^*$-polynomial and the $cd$-index, as summarized in the table above.
%, where the structure sheaf $\cC$ is defined in Section~\ref{sec:cdindex}.
We refer the reader to \cite{Bra06} for a survey of this theory.

Moreover,
the theory of combinatorial intersection cohomology
admits a \emph{decomposition theorem} (Theorem~\ref{t:decomptheorem}),
proved in \cite[Theorem~2.4]{BBFK02} and \cite[Theorem~5.3]{BL03}.
It states that 
a pure sheaf can be decomposed into a direct sum of shifts of simple sheaves
and that 
the direct image of a pure sheaf under a fan subdivision is again a pure sheaf.
This is analogous to the decomposition theorem proved by Beilinson, Bernstein, Deligne and Gabber \cite[Theorem~6.2.5]{BBD82} specialized to proper birational morphisms of toric varieties; see \cite{dCM09} for a survey on the theorem and see \cite{dCMM18} for a study of the theorem for proper toric morphisms.
In particular, the decomposition theorem for the combinatorial invariants can be deduced from the decomposition theorem for pure sheaves, 
where the invariant is given by the \textit{Poincar\'e polynomial}
and the local invariant is given by 
the \textit{local Poincar\'e polynomial}, 
as defined in Section~\ref{sec:puresheaves}.

We define the filtrations on the combinatorial intersection cohomology as follows.
We first define the \emph{weight sheaves} $\cW^r \cF$,
a decreasing filtration of subsheaves,
on a pure sheaf $\cF$ with structure sheaf $\cA$ on $\Delta$.
They naturally induce a \textit{decreasing} filtration $W^\bullet$ on $\ov{\cF(\Delta)}$, which we call the \emph{weight filtration}.
This definition is motivated by the definition of the weight filtration in mixed Hodge theory (see for example \cite[Theorem~4.2]{PS08}),
though we note that in mixed Hodge theory the weight filtration is an increasing filtration.

Now, for a fan subdivision $\pi\colon \Sigma\to \Delta$ and a pure sheaf $\cF$ on $\Sigma$,
the graded vector space $\ov{\pi_*\cF(\Delta)}$ is equipped with the weight filtration $W^\bullet$ from the weight sheaves $\cW^r \pi_* \cF$ of the pure sheaf $\pi_*\cF$ on $\Delta$.
Furthermore, the direct images of the weight sheaves $W^r\cF$ of the pure sheaf $\cF$ on $\Sigma$ induce an additional \textit{decreasing} filtration $M^\bullet$ on $\ov{\pi_* \cF(\Delta)}$,
which we call the \emph{monodromy weight filtration}.
We also define an \textit{increasing} filtration $F^\bullet$,
which we call the \emph{Hodge filtration}, using the grading of the $\Z_{\geq 0}$-graded vector space $\ov{\pi_*\cF(\Delta)}$.
Hence $\ov{\pi_*\cF(\Delta)}$ is equipped with three filtrations.
Its \emph{refined limit Hodge--Deligne polynomial} is defined to be the generating function of the dimensions of the associated gradeds:
\[
E(\ov{\pi_*\cF(\Delta)}, F_\bullet, W^\bullet, M^\bullet; u, v, w) :=
\sum_{p, q, r} \dim_\R
(
\Gr^F_p \Gr_M^{p+q} \Gr_W^{r} \ov{\pi_*\cF(\Delta)}
) u^p v^q w^r.
\]
Then we prove the following result for the \textit{Ehrhart sheaf} $\cE_\Sigma$.
Here, a fan subdivision between \textit{Gorenstein fans} with \textit{equal degree map} 
is the fan analogue 
of 
a lattice polyhedral subdivision of a lattice polyhedral complex; see Section~\ref{sec:Ehrhart} for its definition.

\begin{restatable*}{thm}{resmixedhstar}\label{t:refinedEhrhart}
Let $\pi\colon \Sigma\to \Delta$ be a fan subdivision between quasi-convex Gorenstein fans in $\R^d$ with equal degree map.
Then the refined limit Hodge--Deligne polynomial of the tuple $(\ov{\pi_*\cE_\Sigma(\Delta)}, F_\bullet, W^\bullet, M^\bullet)$ is equal to the refined limit mixed $h^*$-polynomial $h^*(\Delta, \Sigma; u, v, w)$.
\end{restatable*}

If we discard the monodromy weight filtration,
then $\ov{\pi_*\cF(\Delta)}$ is equipped with just the Hodge filtration $F_\bullet$ and the weight filtration $W^\bullet$.
Its \emph{Hodge--Deligne polynomial} is defined to be
\[
E(\ov{\pi_*\cF(\Delta)}, F_\bullet, W^\bullet; u, v) :=
\sum_{p, q} \dim_\R
(
\Gr^F_p \Gr_W^{p+q} \ov{\pi_*\cF(\Delta)}
) u^p v^q.
\]
Then we prove the following result for the \textit{minimal extension sheaf $\cL_\Sigma$ with structure sheaf $\cA$}.

\begin{restatable*}{theorem}{resmixedh}
\label{t:mixedh}
Let $\pi\colon \Sigma\to \Delta$ be a fan subdivision between quasi-convex fans in $\R^d$.
Then the Hodge--Deligne polynomial of the tuple $(\ov{\pi_* \cL_\Sigma(\Delta)}, F_\bullet, W^\bullet)$ is equal to the mixed $h$-polynomial $h_{\Delta}(\Sigma; u, v)$.
\end{restatable*}

In particular, 
this implies the nonnegativity
of the coefficients of the mixed $h$-polynomial.
Furthermore, in the case of a \textit{projective} fan subdivision (the fan analogue of a regular polyhedral subdivision),
using the relative hard Lefschetz theorem proved by Karu \cite[Theorem~1.1]{Kar19},
we show that the coefficients are symmetric and unimodal; see Corollary~\ref{c:mixedhproperties}.
This generalizes a result on the mixed $h$-polynomial by Katz and Stapledon \cite[Corollary~6.7]{KS16a}, which requires the subdivision to be rational.

An analogue of Theorem~\ref{t:mixedh} holds for the mixed $cd$-index.
For a fan subdivision $\pi\colon \Sigma\to \Delta$ in $\R^d$,
the direct image sheaf $\pi_* \cL_\Sigma$ of the minimal extension sheaf $\cL_\Sigma$ with structure sheaf $\cC$ (the pure sheaf that gives the $cd$-index) induces a $(\Z_{\geq 0})^d$-graded vector space $\ov{\pi_*\cL_\Sigma(\Delta)}$.
Since $\ov{\pi_*\cL_\Sigma(\Delta)}$ is $(\Z_{\geq 0})^d$-graded instead of $\Z_{\geq 0}$-graded, we modify the definition of the weight sheaves, by considering 
the \textit{$t$-degree}, 
a specialization of the $(\Z_{\geq 0})^d$-grading:
Let $T\colon (\Z_{\geq 0})^d \to \Z_{\geq 0}$ be an additive map that sends each $e_i \in (\Z_{\geq 0})^d$ to $2^{i-1}$.
For a homogeneous element $f$ of degree $\deg(f) \in (\Z_{\geq 0})^d$, its \textit{$t$-degree} $\tdeg(f)$ is defined to be $T(\deg(f))$.
We then similarly define the weight filtration $W^\bullet$ and the Hodge filtration $F_\bullet$ on $\ov{\pi_*\cL_\Sigma(\Delta)}$,
and prove the following result for the \textit{minimal extension sheaf $\cL_\Sigma$ with structure sheaf $\cC$}.

\begin{restatable*}{theorem}{resmixedcd}\label{t:mixedcd}
Let $\pi\colon \Sigma\to \Delta$ be a fan subdivision between quasi-convex fans in $\R^d$.
Then the Hodge--Deligne polynomial of the tuple $(\ov{\pi_* \cL_\Sigma(\Delta)}, F_\bullet, W^\bullet)$ is equal to 
the image of the mixed $cd$-index $\Omega_{\pi}(c', d', c, d)$ under the injective linear map $\eta'$.
\end{restatable*}

\subsection{Organization of the paper}
The paper is organized as follows:
\begin{itemize}[leftmargin=24pt]
\item In Section~\ref{sec:prelim}, we give the necessary background on polyhedral fans, graded modules and sheaves on a fan.

\item In Section~\ref{sec:puresheaves}, we give an introduction to the theory of pure sheaves on a fan.

\item In Section~\ref{sec:weightsheaves}, we introduce the weight sheaves.

\item In Section~\ref{sec:polysfromfiltrations}, 
we define the filtrations using the weight sheaves.
We then define 
the Hodge--Deligne polynomials using these filtrations.

\item In Section~\ref{sec:hpoly}, we give an introduction to the mixed $h$-polynomial and prove Theorem~\ref{t:mixedh}.

\item In Section~\ref{sec:Ehrhart}, we give an introduction to the refined limit mixed $h^*$-polynomial and prove Theorem~\ref{t:refinedEhrhart}.

\item In Section~\ref{sec:cdindex}, we give an introduction to the mixed $cd$-index.
Then we define the weight sheaves in this setting and prove Theorem~\ref{t:mixedcd}.
\end{itemize}

\medskip
\noindent
{\it Acknowledgments.}
The author would like to thank his advisor Eric Katz, without whom this paper would not have been possible.
The author was partially supported through
his advisor's grant NSF DMS 1748837.

\section{Preliminaries}\label{sec:prelim}

In this section, we give a brief introduction to polyhedral fans, graded modules and sheaves on fans following \cite[Section~1]{BBFK02}.

\subsection{Fans}\label{subsec:fans}

A \textit{cone} is a subset $\sigma$ of $\R^d$ such that
$\sigma = \R_{\geq 0} v_1+ \dots+ \R_{\geq 0} v_n$ for some $v_1, \dots, v_n \in \R^d$.
The dimension of $\sigma$ is the dimension of the linear subspace spanned by $\sigma$.
Let $o:= \{0\}$ be the cone containing only the origin. 

A face of a cone $\sigma$ is the cone given by $\{u\in \sigma : w(u)=0 \}$, where $w\in (\R^d)^*$ is a linear form satisfying $w(\sigma)\geq 0$.
We write $\tau \leq  \sigma$ if $\tau$ is a face of $\sigma$ and $\tau < \sigma$ if $\tau$ is a proper face of $\sigma$.
A cone is \textit{pointed} (also known as strongly convex) if it does not contain a line through the origin.

A \textit{fan} is a finite collection $\Delta$ of pointed cones in $\R^d$ such that
\begin{enumerate}
\item \label{con:fan1}
every face of a cone in $\Delta$ is also contained in $\Delta$; and

\item \label{con:fan2}
the intersection of any two cones in $\Delta$ is a face of both cones.
\end{enumerate}
The dimension of a fan is the maximum of the dimensions of its cones.
Given a cone $\sigma$,
we let $\langle \sigma\rangle$ be the fan generated by $\sigma$, that is, the collection of all the faces of $\sigma$.

\begin{notation}
From now on, every cone is assumed to be pointed.
\end{notation}

A fan is \textit{simplicial} if 
every cone $\sigma$ in the fan 
can be generated by $\dim(\sigma)$ many vectors.

A fan is \textit{purely $n$-dimensional} if every maximal cone is of dimension $n$.

The support of a fan $\Delta$, denoted by $|\Delta|$, is defined to be the union of its cones.
A fan in $\R^d$ is \textit{complete} if the support is the whole space $\R^d$.
A complete fan $\Delta$ is \textit{projective} if it admits a strictly convex conewise linear function $\ell$.
We note that a complete fan is projective if and only if it is the face fan of a $d$-dimensional polytope that contains the origin in the interior.

A subfan $U$ of a fan $\Delta$ is a subset of $\Delta$ that is also a fan.
The \textit{boundary fan} $\partial \Delta$ of a fan $\Delta$ is the subfan generated by the non-maximal cones contained in exactly one maximal cone.
Given a cone $\sigma$, we denote by $\partial \sigma$ the boundary fan of $\langle \sigma \rangle$.
We note that if a fan $\Delta$ is purely $n$-dimensional, 
then its boundary fan is supported on the topological boundary of $|\Delta|$.

For a cone $\sigma$ in a fan $\Delta$, we define the \textit{link} of $\sigma$,
which we denote by $\lk_\Delta \sigma$,
as follows.
Let $p\colon \R^d \to \R^d/ \Span(\sigma)$ be the quotient map. 
Given a cone $\tau \geq \sigma$, 
we denote the image of $\tau$ under $p$ by $\ov{\tau}$.
Then the link $\lk_\Delta \sigma$ is defined to be the fan
\[
\lk_\Delta \sigma := \{ \ov{\tau} : \tau \geq \sigma\}.
\]

\begin{defn}[{\cite[Theorem~4.4]{BBFK02}}]
A purely $d$-dimensional fan $\Delta$ in $\R^d$ is \emph{quasi-convex} if $|\partial \Delta|$ is a real homology manifold (see for example \cite{Mio00} for its definition).
\end{defn}

Examples of quasi-convex fans include complete fans (which has empty boundary), fans generated by a full-dimensional cone, and fans where the supports or the complements of the supports are full-dimensional convex sets.

\begin{prop}[{\cite[Theorem~4.3]{BBFK02}}]\label{p:linkofqcisqc}
Let $\Delta$ be a quasi-convex fan.
Then for every $\sigma \in \Delta$,
the link $\lk_\Delta \sigma$ is again a quasi-convex fan.
\end{prop}

We turn our attention to subdivisions of fans. 

\begin{defn}
A fan $\Sigma$ is a \textit{refinement} of a fan $\Delta$ if every cone of $\Sigma$ is contained in some cone of $\Delta$ and $|\Sigma| = |\Delta |$.
A set function $\pi\colon \Sigma \to \Delta$ is a \emph{fan subdivision} if $\Sigma$ is a refinement of $\Delta$ and $\pi$ sends each $\sigma\in \Sigma$ to the smallest cone in $\Delta$ containing $\sigma$.
\end{defn}

A fan subdivision $\pi\colon \Sigma \to \Delta$ is  \textit{projective} if it admits a \textit{relatively strictly convex function} with respect to $\pi$ (see \cite[p.~861]{Kar19} for its definition).
If $\mathcal{S}$ is a regular polyhedral subdivision of a polytope $P$
(see \cite[Section~2.3]{DLRS10} and \cite[Section~1.F]{GB09} for equivalent definitions of a regular polyhedral subdivision),
then the induced fan subdivision is projective.

We introduce \textit{proper fan morphisms},
a more general notion of fan maps, as follows.
A \textit{fan morphism} $\pi\colon \Sigma\to \Delta$ \cite[Section~1.4]{Ful93} is a map induced by an underlying linear map $P\colon \R^{d_1} \to \R^{d_2}$, such that for a cone $\sigma\in \Sigma$, its image $\pi(\sigma)$ is the smallest cone in $\Delta$ containing $P(\sigma)$.
In particular, a fan subdivision is a fan morphism induced by the identity map.
A fan morphism is \emph{proper} \cite[Section~2.4]{Ful93} if $P^{-1}(|\Delta|) = |\Sigma|$.

\subsection{Graded modules}\label{subsec:graded}

The results in this section can be found in
\cite[Section~11]{AM69}
and
\cite{BH93}.

For a $\Z_{\geq 0}$-graded vector space $V = \oplus_{i\in \Z_{\geq 0}} V^i$, we define the \textit{Hilbert series} of $V$ to be
\[\Hilb(V; t) = \sum_{i\in \Z_{\geq 0}} \dim_\R(V^i) t^i.\]
We write $V^{\geq i}$ and $V^{\leq i}$ for $\oplus_{j\geq i} V^j$ and $\oplus_{j\leq i} V^j$, respectively.
For an integer $j$, we define the \textit{shifted graded vector space} $V[-j]$ by $(V[-j])^{i} = V^{i-j}$.

Let $A= \R[x_1, \dots, x_d]$ be the graded ring of polynomial functions on $\R^d$, where $\deg(x_i)  = 1$ for every $i$.
Then $\mmax := (x_1, \dots, x_d)$ is the homogeneous maximal ideal of $A$.

\begin{defn}
Let $M$ be a finitely generated graded $A$-module.
We define $\ov{M}$ to be the graded vector space $M\otimes_A A/\mmax = M/\mmax M$ and $\rho\colon M\to \ov{M}$ to be the corresponding quotient map.
\end{defn}

We observe that
for each $k \in \Z_{\geq 0}$, 
the $A$-submodule $M^{\geq k}$ is sent to the subspace $\ov{M}^{\geq k}$ under $\rho$:
\begin{equation}\label{eq:SubmoduleSentToSubspace}
\rho(M^{\geq k}) = \ov{M}^{\geq k}.
\end{equation}
%This fact is needed for the definition of the weight filtration in Section~\ref{sec:polysfromfiltrations}.

%We may write $\rho$ for $\rho_M$ when the module is clear.

For a finitely generated graded $A$-module $M$, we define the \emph{Poincar\'e polynomial} of $\ov{M}$ to be
\[
P(\ov{M}; t) = \sum_{i\in \Z_{\geq 0}} \dim_{\R} (\ov{M}^i) t^i.
\]
If $M$ is a free $A$-module, then 
\begin{equation}\label{eq:freeimplieshilbpoin}
\Hilb(M; t) = \frac{P(\ov{M}; t)}{(1-t)^d}.
\end{equation}

\subsection{Sheaves on fans}\label{subsec:sheaves}

Let $\Delta$ be a fan in $\R^d$.
We put a topology on $\Delta$ as follows:
the set of open sets is defined to be the set of subfans of $\Delta$.
Hence 
the subfans $\langle\sigma \rangle$ generated by a single cone 
form a basis of the topology.

To define a sheaf $\cF$ on $\Delta$, it suffices to define the stalk $\cF(\sigma)$ at every cone $\sigma$ and the map $\res_{\sigma, \tau}\colon \cF(\sigma) \to \cF(\tau)$ for every pair of cones $\sigma \geq \tau$ such that
$\res_{\tau, \gamma}\circ \res_{\sigma, \tau}  = \res_{\sigma, \gamma}$
whenever $\sigma \geq \tau \geq \gamma$.
Then, for a general open set $U$, the set $\cF(U)$ is given by
\[
\cF(U) = \{ (f_\sigma)_{\sigma\in U} : f_\sigma \in \cF(\sigma) ,
\res_{\sigma, \tau}(f_{\sigma}) = f_\tau \text{ for all } \sigma \leq \tau
\},
\]
and for a pair of open sets $U\supseteq V$, the restriction map $\res\colon \cF(U)\to \cF(V)$ is given by forgetting the function $f_\sigma$ whenever $\sigma\notin V$.
In particular, we have $\cF(\langle \sigma \rangle) = \cF(\sigma)$, hence we may use these two notations interchangeably.

For every cone $\sigma \in \Delta$, we let $A_\sigma$ be the $\Z_{\geq 0}$-graded ring of polynomial functions on $\Span(\sigma)$,
where a linear function is homogeneous of degree $1$.
We define the \emph{structure sheaf} $\cA$ of a fan $\Delta$ as follows.
On each cone $\sigma$, we set $\cA(\sigma) := A_\sigma$.
For cones $\sigma \geq \tau$, the restriction map $\res_{\sigma, \tau} \colon \cA(\sigma) \to \cA(\tau)$ is the graded ring map
given by restricting the domain of a polynomial function.
This defines a sheaf of graded rings.

A sheaf $\cF$ is a \textit{sheaf of $\Z_{\geq 0}$-graded $\cA$-modules} if $\cF(U)$ is a $\Z_{\geq 0}$-graded $\cA(U)$-module whenever $U$ is an open set and the restriction maps of $\cF$ are graded module maps
compatible with that of $\cA$.
Recall that $A = \R[x_1, \dots, x_d]$ is the graded ring of polynomial functions on $\R^d$.
Hence each $\cF(U)$ is a graded $A$-module and each restriction map of $\cF$ is a graded $A$-module map.
%since for each cone $\sigma$
%there is a canonical ring map $A\to A_\sigma$.

A sheaf $\cF$ is \textit{flabby} if every restriction map is surjective.
To show the flabbiness of $\cF$, it suffices to show that $\res\colon \cF(\sigma) \to \cF(\partial \sigma)$ is surjective whenever $\sigma$ is a cone.
We note that the structure sheaf $\cA$ is in general not flabby.
Indeed, it is flabby if and only if the fan $\Delta$ is simplicial \cite[Proposition~1.4]{BBFK02}.

For notational convenience, 
we write 
$\cF(\Delta, \partial \Delta)$
for 
$\ker(\res \colon \cF(\Delta) \to  \cF(\partial \Delta))$.
For a cone $\sigma\in \Delta$,
we write 
$\cF(\sigma, \partial \sigma)$
for 
$\ker(\res \colon \cF(\sigma) \to  \cF(\partial \sigma))$.

For a sheaf $\cF$ and an integer $j\in \Z_{\geq 0}$, the \textit{shifted sheaf} $\cF[-j]$ is defined by 
\[
\cF[-j](U) = \cF(U)[-j]
\]
whenever $U$ is an open set.
%equipped with the shifted restriction maps of $\cF$.

Suppose $\pi\colon \Sigma\to \Delta$ is a fan subdivision. 
This is a continuous map. 
Thus given a sheaf $\cF$ of $\Z_{\geq 0}$-graded $\cA_\Sigma$-modules 
on $\Sigma$,
the \textit{direct image sheaf} $\pi_* \cF$ 
is 
a sheaf of $\Z_{\geq 0}$-graded $\cA_\Delta$-modules 
on $\Delta$ defined by
\[
\pi_* \cF(U) := \cF(\pi^{-1}(U))
\]
whenever $U\subseteq \Delta$ is an open set.

\section{Pure sheaves on a fan}\label{sec:puresheaves}

In this section we summarize constructions and results from \cite{BBFK02, Bra06, BL03, Kar04, Kar19}.

\begin{defn}\label{defn:puresheaf}
A sheaf $\cF$ of $\Z_{\geq 0}$-graded $\cA$-modules on a fan $\Delta$ is called a \emph{pure sheaf} if $\cF$ is flabby and $\cF(\sigma)$ is a finitely generated free $A_\sigma$-module whenever $\sigma$ is a cone.
\end{defn}

If $\cF$ is a pure sheaf on a fan $\Delta$,
then both $\cF(\Delta)$ and $\cF(\Delta, \partial \Delta)$ are finitely generated as $A$-modules \cite[p.~8]{BBFK02}.
Hence both 
$P(\ov{\cF(\Delta)}; t)$
and 
$P(\ov{\cF(\Delta, \partial \Delta)}; t)$
are well-defined. 

\begin{defn}\label{defn:SimpleSheaves}
For a cone $\sigma$ in a fan $\Delta$, a \emph{simple sheaf} $\cL_\sigma$ based at $\sigma$ is a sheaf of $\Z_{\geq 0}$-graded $\cA$-modules satisfying the following conditions:
\begin{itemize}
\item \textit{Normalization:} For each $\tau\in \Delta$ of dimension at most $\dim(\sigma)$,
\[
\cL_\sigma(\tau) = \begin{cases}
A_\sigma & \text{if } \tau  = \sigma,\\
0 & \text{otherwise.}
\end{cases}
\]

\item \textit{Local freeness:} For each cone $\tau \in \Delta$, the $A_\tau$-module $\cL_\sigma(\tau)$ is free.

\item \textit{Local minimal extension:} For each cone $\tau \in \Delta$ of dimension strictly greater than $\dim(\sigma)$, the restriction map $\res\colon \cL_\sigma(\tau) \to\cL_\sigma(\partial \tau)$ induces an isomorphism of graded vector spaces:
\[
\begin{tikzcd}
\ov{\res}\colon \ov{\cL_\sigma(\tau)}
\arrow{r}{\sim}
& \ov{\cL_\sigma(\partial \tau)}.
\end{tikzcd}
\]
\end{itemize}
\end{defn}

For each cone $\sigma$, the simple sheaf $\cL_\sigma$ exists 
and is unique up to an isomorphism \cite[Proposition~1.3]{BBFK02}.
We call the simple sheaf $\cL_o$ based at the zero cone $o$ the \emph{minimal extension sheaf} $\cL_\Delta$ of $\Delta$ \cite[Definition~1.1]{BBFK02}.

Clearly, a simple sheaf is a pure sheaf.
Furthermore, the simple sheaves are the simple objects in the category of pure sheaves, as summarized in the following decomposition theorem for pure sheaves proved in \cite[Theorem~2.4]{BBFK02} and \cite[Theorem~5.3]{BL03}.

\begin{thm}\label{t:decomptheorem}
Every pure sheaf $\cF$ on a fan $\Delta$ admits a direct sum decomposition into shifts of simple sheaves:
\[
\cF \simeq \bigoplus_{\sigma\in \Delta} K_\sigma\otimes_\R
\cL_{\sigma},
\]
where $K_\sigma := \ker(\ov{\res}\colon \ov{\cF(\sigma)}\to \ov{\cF(\partial \sigma)})$ is a $\Z_{\geq 0}$-graded vector space.
\end{thm}

In general, 
the isomorphism in the decomposition is non-unique;
see \cite[Section~4.1]{Kar19} for a discussion. 
Alternatively we can write $\cF \simeq \bigoplus_\alpha \cL_{\sigma_\alpha}[-j_\alpha]$,
where each $\sigma_\alpha$ is a possibly repeated cone in $\Delta$ and each $j_\alpha$ is a possibly repeated integer.
We define the \emph{local Poincar\'e polynomial} of $\cF$ at $\sigma$ to be
\begin{equation*}\label{eq:localPoincar\'epoly}
L(\cF, \sigma; t) := P(K_\sigma; t) = \sum_{
\substack{\alpha\\ \sigma_\alpha = \sigma}
} t^{j_\alpha}.
\end{equation*}
Thus the theorem implies
\begin{equation}\label{eq:Poincar\'erecursion}
P(\ov{\cF(\Delta)}; t) = \sum_{\sigma\in \Delta} L(\cF, \sigma; t) \cdot P(\ov{\cL_\sigma(\Delta)}; t).
\end{equation}

\subsection{Pure sheaves on a quasi-convex fan}\label{subsec:pureonquasi}

We begin with a characterization of the quasi-convex fans \cite[Definition~4.1]{BBFK02}; see also \cite[Proposition~6.1]{BBFK99}.
We note that while the following statement is slightly stronger than the original statement, 
they are equivalent due to \cite[Theorem~4.3(c)]{BBFK02}.

\begin{prop}\label{p:quasiconvex}
Let $\Delta$ be a purely $d$-dimensional fan in $\R^d$.
Then $\Delta$ is quasi-convex if and only if 
the $A$-module $\cF(\Delta)$ is free
whenever $\cF$ is a pure sheaf. 
\end{prop}

In particular,
the proposition implies that
the direct image sheaf of a pure sheaf $\cF$ under a fan subdivision is again a pure sheaf:
Let $\pi\colon \Sigma\to \Delta$ be a fan subdivision
and $\cF$ be a pure sheaf on $\Sigma$.
Then the direct image sheaf $\pi_* \cF$ is flabby,
since $\cF$ is. 
Now, the preimage fan $\pi^{-1}(\langle \sigma \rangle)$ has support $\sigma$, hence it is a 
quasi-convex fan in $\Span(\sigma)$.
By the proposition, the module $\pi_*\cF(\sigma) = \cF(\pi^{-1}(\langle \sigma \rangle))$ is a free $A_\sigma$-module, 
hence $\pi_*\cF$ is a pure sheaf on $\Delta$.

We turn our attention to the minimal extension sheaf $\cL_\Delta$ of a quasi-convex fan $\Delta$ in $\R^d$.
By the previous proposition, the module $\cL_\Delta(\Delta)$ is a free $A$-module. In addition, the module $\cL_\Delta(\Delta, \partial \Delta) = \ker(\res\colon \cL_\Delta(\Delta) \to \cL_\Delta(\partial \Delta))$ is also a free $A$-module \cite[Corollary~4.12]{BBFK02}, and the duality theory developed in \cite[Section~6]{BBFK02} implies that they are dual to each other in the following sense:
\[
\cL_\Delta(\Delta) =  \Hom_A(\cL_\Delta(\Delta, \partial \Delta), A[-d]).
\]
See also \cite{BBFK05, BL05} for a canonical treatment of the duality theory.

\begin{thm}[Hard Lefschetz {\cite[Theorem~0.1]{Kar04}}]\label{t:hardlef}
Let $\Delta$ be a projective fan in $\R^d$ and $\ell$ be a conewise linear strictly convex function on $\Delta$.
Then multiplication of $\ell^{d-2k}$
induces an isomorphism of vector spaces for each integer $k\leq d/2$:
\[
\begin{tikzcd}
\ell^{d-2k} \colon \ov{\cL_\Delta(\Delta)}^{k}  \arrow{r}{\sim} &  \ov{\cL_\Delta(\Delta)}^{ d-k}.
\end{tikzcd}
\]
\end{thm}

Given a cone $\tau$, the boundary fan $\partial \tau$ is isomorphic to a projective fan.
Thus the hard Lefschetz theorem implies the following degree vanishing condition for any simple sheaf $\cL_\sigma$, as stated in \cite[Theorem~2.6]{Bra06}.

\begin{cor}\label{c:bradenstalk}
Let $\Delta$ be a fan.
Let $\cL_\sigma$ be the simple sheaf based at $\sigma \in \Delta$ and $\tau > \sigma$ be a cone in $\Delta$.
Then we have the following:
\begin{enumerate}
\item The module $\cL_\sigma(\tau)$ is generated in degrees strictly less than $(\dim(\tau) - \dim(\sigma)) / 2$.

\item The module $\cL_\sigma(\tau, \partial \tau) = \ker(\res\colon \cL_\sigma(\tau)\to \cL_\sigma(\partial \tau))$ is generated in degrees strictly greater than $(\dim(\tau) - \dim(\sigma)) / 2$.
\end{enumerate}
\end{cor}

In particular, the Poincar\'e polynomial $P(\ov{\cL_\Delta(\tau)}; t)$ is of degree strictly less $\dim(\tau)/2$ whenever $\dim(\tau) > 0$.

\begin{thm}[Relative hard Lefschetz {\cite[Theorem~1.1]{Kar19}}]\label{t:relhardlef}
Let $\pi\colon \Sigma \to \Delta$ be a projective fan subdivision between fans and
$\hat{\ell}$ be a relatively strictly convex
function with respect to $\pi$.
Let $\pi_*\cL_\Sigma = \oplus_{\sigma\in \Delta} K_\sigma\otimes_\R \cL_\sigma$ be a decomposition of the pure sheaf $\pi_*\cL_\Sigma$ on $\Delta$.
Then multiplication of $\hat{\ell}^{\dim(\sigma)-2k}$
induces an isomorphism of vector spaces for each $k\leq \dim(\sigma)/2$:
\[
\begin{tikzcd}
\hat{\ell}^{\dim(\sigma)-2k} \colon K_\sigma^{ k}  \arrow{r}{\sim} &  K_\sigma^{\dim(\sigma)-k}.
\end{tikzcd}
\]
\end{thm}

In particular, the coefficients of the local Poincar\'e polynomial $L(\pi_* \cL_\Sigma, \sigma; t) = P(K_\sigma; t)$ are 
symmetric about degree $\dim(\sigma)/2$
and are unimodal.

\section{Weight sheaves}\label{sec:weightsheaves}

In this section we define the weight sheaves $\cW^r \cF$ of a pure sheaf $\cF$,
motivated by the definition of the weight filtration in mixed Hodge theory (see for example \cite[Theorem~4.2]{PS08}).

\begin{defn}
Let $\cF$ be a pure sheaf on a fan $\Delta$.
For every integer $r\in \Z$, we define an $\cA$-module subsheaf $\cW^r \cF$ of $\cF$, called the \emph{$r$-weight sheaf} of $\cF$ as follows:
for each cone $\sigma$, we set
\[
\cW^r\cF(\sigma) = \begin{cases}
\cF(\sigma) & \text{if } r\leq 0, \\
\res_\cF^{-1}(\cW^{r} \cF(\partial \sigma)) & \text{if } 0 < r \leq \dim (\sigma), \\
\mmax \cW^{r-2}\cF(\sigma)  & \text{if } r> \dim (\sigma).\\
\end{cases}
\]
\end{defn}

In particular, on an open set $U$, we have
\[
\cW^r\cF(U) \simeq \{ (f_\sigma)_{\sigma \in U} :
f_\sigma\in \cW^r\cF(\sigma) ,
\res_{\sigma, \tau}(f_{\sigma}) = f_\tau \text{ for all } \tau \leq \sigma
\}.
\]
We prove that the weight sheaves are indeed well-defined subsheaves of $\cF$ in Proposition~\ref{p:weightissheaf}.
For now, 
we consider each weight sheaf an assignment to 
each open set $U$ 
of 
an $A(U)$-submodule of $\cF(U)$.

\begin{prop}\label{p:madds2}
For every open set $U$ and every integer $r$, we have $\mmax \cW^{r-2}\cF(U) \subseteq \cW^{r}\cF(U)$.
\end{prop}

\begin{proof}
We want to prove
\[
(f_\sigma)_{\sigma\in U} \in \cW^{r-2}\cF(U) \implies
\forall a\in \mmax, \,
(a  f_\sigma)_{\sigma\in U} \in \cW^{r}\cF(U).
\]
It suffices to prove the statement for $U = \langle \sigma \rangle$ whenever $\sigma$ is a cone, which we do by induction on $\dim(\sigma)$.
The statement is clear for the zero cone $o$.
By induction we may assume the statement for $U = \partial \sigma$.

The case when $r\leq 0$ or $r> \dim(\sigma)$ is clear.
Suppose $0 < r \leq \dim(\sigma)$.
Then 
\[
\res(\mmax \cW^{r-2}\cF(\sigma))
=
\mmax \res(\cW^{r-2}\cF(\sigma))
=
\begin{cases}
\mmax \res(\cF(\sigma))
& \text{if } r = 1, 2\\
\mmax \res(\res^{-1}(\cW^{r-2}\cF(\partial \sigma)))
& \text{otherwise.}
\end{cases}
\]
Since $\res$ is surjective,
we have $\res(\cF(\sigma)) = \cF(\partial \sigma) = \cW^{r-2} \cF(\partial \sigma)$ when $r = 1, 2$.
Hence for $0 < r \leq \dim(\sigma)$
we have
\[
\res(\mmax \cW^{r-2}\cF(\sigma))
=
\mmax \cW^{r-2}\cF(\partial \sigma).
\]
By the induction hypothesis, we also have $\mmax \cW^{r-2} \cF(\partial \sigma) \subseteq \cW^{r}\cF(\partial \sigma)$.
Now, 
by taking the preimage,
we have 
\begin{align*}
\mmax \cW^{r-2}\cF(\sigma)
& \subseteq 
\res^{-1}
\big(
\res
(
\mmax \cW^{r-2}\cF(\sigma)
)
\big) \\
& \subseteq 
\res^{-1}
\big(
\mmax
 \cW^{r-2}\cF(\partial \sigma)
\big) \\
& \subseteq 
\res^{-1}
\big(
\cW^r \cF(\partial \sigma)
\big) \\
& =
\cW^r \cF(\sigma).
\qedhere
\end{align*}
\end{proof}

We check the well-definedness of the weight sheaves.

\begin{prop}\label{p:weightissheaf}
The assignment $\cW^r\cF$ together with the restriction maps of $\cF$ form a sheaf of $\Z_{\geq 0}$-graded $\cA$-modules.
\end{prop}

\begin{proof}
It remains to show that the restriction map $\res_\cF\colon \cF(\sigma) \to \cF(\partial \sigma)$ satisfies
\[
\res_\cF(\cW^r\cF(\sigma)) \subseteq \cW^r\cF(\partial \sigma)
\]
for every cone $\sigma$ and every integer $r$,
which we prove by induction on $r$.
This is clear if $r\leq \dim(\sigma)$.
For $r>\dim (\sigma)$, we have
\begin{align*}
\res_{\cF}(\mmax \cW^{r-2}\cF(\sigma)) =
 \mmax \res_{\cF}(\cW^{r-2}\cF(\sigma))
& \subseteq \mmax\cW^{r-2}\cF(\partial \sigma)\\
& \subseteq\cW^r\cF(\partial \sigma),
\end{align*}
where the first inclusion follows from the induction hypothesis and the second inclusion follows from Proposition~\ref{p:madds2}.
\end{proof}

We prove a characterization of the weight sheaves of a shifted simple sheaf.

\begin{prop}\label{p:charofweight}
Let $\cF = \cL_\sigma[-j]$ be a shifted simple sheaf on a fan $\Delta$, where $j\in \Z_{\geq 0}$.
Then for every open set $U$ and every integer $r$ we have
\begin{equation}
\label{eq:charofw}
\cW^r\cF(U) =
\cF(U)^{\geq j+\frac{r - \dim(\sigma)}{2}}.
\end{equation}
\end{prop}

\begin{proof}
We want to prove that for every open set $U$ we have
\[
(f_\tau)_{\tau\in U} \in \cW^r\cF(U) \iff \forall \tau\in U, \,
\text{either } f_\tau = 0  \text{ or }
\deg(f_\tau)\geq j+\frac{r - \dim(\sigma)}{2}.
\]
Hence it suffices to prove \eqref{eq:charofw} for $U = \langle \tau \rangle$ whenever $\tau \geq \sigma$ is a cone, which we do by induction on $\dim(\tau)$.

When $\tau = \sigma$, 
by the definition of the weight sheaves we have
\[
\cW^r\cF(\sigma) = \begin{cases}
A_\sigma[-j] & \text{if } r\leq \dim(\sigma), \\
\mmax^{\ceil{\frac{r-\dim(\sigma)}{2}}} A_\sigma[-j] & \text{if }  r>\dim(\sigma).
\end{cases}
\]
Since $A_\sigma$ is isomorphic to $\R[x_1, \dots, x_{\dim(\sigma)}]$, 
we have 
\[
\mmax^{\ceil{\frac{r-\dim(\sigma)}{2}}} A_\sigma[-j] 
=
A_\sigma[-j]^{\geq j + \ceil{\frac{r-\dim(\sigma)}{2}}
}
=
A_\sigma[-j]^{\geq j + \frac{r-\dim(\sigma)}{2}}.
\]
Hence \eqref{eq:charofw} holds for $U = \langle \sigma\rangle$.

Now, for a cone $\tau > \sigma$,
by induction we may assume \eqref{eq:charofw} for $U = \partial \tau$.
We further induct on $r$,
and prove the base case as follows. 
The case when $r\leq 0$ is clear.
Suppose $0 < r\leq \dim(\tau)$.
Since the restriction map is homogeneous of degree $0$
and 
$\cW^r \cF(\partial \tau) = \cF(\partial \tau)^{\geq j + \frac{r-\dim(\sigma)}{2}}$
by the induction hypothesis, we have
\begin{align*}
\cW^r\cF(\tau) = \res^{-1}(\cW^r\cF(\partial \tau))
& = \cF(\tau, \partial \tau)
+
\cF(\tau)^{\geq j+\frac{r - \dim(\sigma)}{2}}\\
& = 
\cF(\tau)^{\geq j+\frac{r - \dim(\sigma)}{2}},
\end{align*}
where the last equality follows from  Corollary~\ref{c:bradenstalk}.

Suppose $r> \dim (\tau)$.
By Corollary~\ref{c:bradenstalk}, the free $A_\tau$-module
$\cF(\tau)$ is generated in degrees strictly less than
$j + \frac{\dim(\tau)-\dim(\sigma)}{2}$.
Thus we have
\[
\cF(\tau)^{\geq j+\frac{r- \dim(\sigma)}{2}}
\subseteq 
\mmax \cF(\tau)^{\geq j+\frac{r-2 - \dim(\sigma)}{2}}.
\]
By comparing the degrees, one can see that the other inclusion also holds.
%since the multiplication of $\mmax$ increases the degree by at least $1$.
By the induction hypothesis of the second induction, we have 
$\cW^{r-2} \cF(\tau)
=
\cF(\tau)^{\geq j+\frac{r-2- \dim(\sigma)}{2}}$. 
Hence 
\begin{align*}
\cW^r\cF(\tau) 
= 
\mmax 
\cW^{r-2}\cF(\tau) 
& =
\mmax \cF(\tau)^{\geq j+\frac{r-2 - \dim(\sigma)}{2}}\\
& =
\cF(\tau)^{\geq j+\frac{r- \dim(\sigma)}{2}}.
\end{align*}
This finishes both inductions. 
\end{proof}

We proceed to prove some properties of the weight sheaves.

\begin{prop}\label{p:WeightSheavesAreDecreasing}
Let $\cF$ be a pure sheaf on a fan $\Delta$.
Then the weight sheaves form a decreasing sequence of subsheaves of $\cF$, i.e.,
for each open set $U$ and $r\in \Z$ 
we have
$\cW^r\cF(U) \supseteq \cW^{r+1}\cF(U)$.
\end{prop}

\begin{proof}
We fix a decomposition of $\cF$ into shifts of simple sheaves:
\[
\cF \simeq \oplus_\alpha \cL_{\sigma_\alpha}[-j_\alpha].
\]
Since the construction of the weight sheaves commutes with direct sums, 
for each $r \in \Z$ we have
\[
\cW^r\cF\simeq \oplus_\alpha 
\cW^r
\cL_{\sigma_\alpha}[-j_\alpha].
\]
Hence it suffices to prove the statement for a shifted simple sheaf $\cL_{\sigma_\alpha}[-j_\alpha]$.
By Proposition~\ref{p:charofweight},
for each open set $U$ we have 
$\cW^r
\cL_{\sigma_\alpha}[-j_\alpha](U) = \cL_{\sigma_\alpha}[-j_\alpha](U)^{\geq j_\alpha
+ \frac{r-\dim(\sigma_\alpha)}{2}}$, 
which is naturally a decreasing filtration.
\end{proof}

We end the section with the functoriality of the weight sheaves.
This result is not needed for the main results of the paper. 

\begin{prop}
Every map $\varphi\colon \cF\to \cG$ of pure sheaves sends $\cW^r \cF$ to $\cW^r\cG$.
\end{prop}

\begin{proof}
We want to prove
\[
\varphi(\cW^{r}\cF(U)) \subseteq \cW^r \cG(U)
\]
whenever $U$ is an open set.
It suffices to prove the statement for $U = \langle \sigma \rangle$ for any cone $\sigma$, which we do by induction on $\dim(\sigma)$.
The statement is clear when $\sigma = o$.
By induction we may assume the statement for $U = \partial \sigma$.

We further induct on $r$.
The case when $r\leq 0$ is clear.
If $0<r\leq \dim(\sigma)$, then $f\in \cW^r \cF(\sigma)$ is sent to $\varphi_\sigma(f)\in \cG(\sigma)$, which restricts to $\varphi_\sigma(f)|_{\partial \sigma} = \varphi_{\partial \sigma}(f|_{\partial \sigma}) \in \cW^r \cG(\partial \sigma)$, by the induction hypothesis of the first induction.
In other words, $\varphi_\sigma(f)$ is contained in $\res^{-1}(\cW^r \cG(\partial \sigma)) = \cW^r \cG(\sigma)$.

If $r> \dim(\sigma)$, then 
\begin{align*}
\varphi(\cW^{r}\cF(\sigma)) = \varphi(\mmax \cW^{r-2}\cF(\sigma)) 
& = \mmax \varphi(\cW^{r-2}\cF(\sigma)) \\
& \subseteq \mmax \cW^{r-2} \cG(\sigma) \\
& = \cW^{r} \cG(\sigma),
\end{align*}
where the inclusion uses the induction hypothesis of the second induction.
This finishes both inductions.
\end{proof}

\section{Hodge--Deligne polynomials}\label{sec:polysfromfiltrations}

In this section, we define the Hodge--Deligne polynomials
using the filtrations from the weight sheaves. They are motivated by polynomials of the same names in mixed Hodge theory; see for example \cite[p.~5]{KS16b}.

\begin{defn}
Let $\cF$ be a pure sheaf on a fan $\Delta$.
The \emph{Hodge filtration} $F_\bullet$ of $\cF$ is defined to be the increasing filtration on
$\ov{\cF(\Delta)}$ given by
\[
F_p := \ov{\cF(\Delta)}^{\leq p}.
\]
\end{defn}

The \textit{associated graded vector space} $\Gr^F_p \ov{\cF(\Delta)}$ is given by
\[
\Gr^F_p \ov{\cF(\Delta)} =
\ov{\cF(\Delta)}^{\leq p}
/
\ov{\cF(\Delta)}^{\leq p-1} =
\ov{\cF(\Delta)}^{p},\]
hence we also have
$P(\ov{\cF(\Delta)}; t) = \sum_p
\dim_\R (\Gr^F_p
\ov{\cF(\Delta)}
)
t^p$.

Recall from Section~\ref{sec:weightsheaves} the definition of the weight sheaves.

\begin{defn}
Let $\cF$ be a pure sheaf on a fan $\Delta$.
The \emph{weight filtration} $W^\bullet$ of $\cF$ is defined to be the decreasing filtration on $\ov{\cF(\Delta)}$ given by
\[
W^r := \rho(\cW^r \cF(\Delta)),
\]
where $\rho\colon \cF(\Delta)\to \ov{\cF(\Delta)}$ is the quotient map.
\end{defn}

The weight filtration is a well-defined filtration, since by Proposition~\ref{p:WeightSheavesAreDecreasing} the weight sheaves are a decreasing sequence of subsheaves.

Both filtrations commute with direct sums:
for every decomposition $\cF \simeq \oplus_\alpha \cL_{\sigma_\alpha}[-j_\alpha]$ of $\cF$ into shifts of simple sheaves, the Hodge (resp. weight) filtration of $\cF$ is the direct sum of the Hodge (resp. weight) filtrations of the shifted simple sheaves.

\begin{defn}\label{defn:HodgeDelignePoly}
Let $\cF$ be a pure sheaf on a fan $\Delta$.
The \emph{Hodge--Deligne polynomial} of $(\ov{\cF(\Delta)}, F_\bullet, W^\bullet)$ is defined to be
\[
E(\ov{\cF(\Delta)}, F_\bullet, W^\bullet; u, v) =
\sum_{p, q} \dim_\R
(
\Gr^F_p \Gr_W^{p+q} \ov{\cF(\Delta)}
)
u^p v^q.
\]
\end{defn}

By specializing $u\mapsto t$ and $v\mapsto 1$, we recover the Poincar\'e polynomial
of
$\ov{\cF(\Delta)}$.

\begin{lem}\label{lem:hodgemixedh}
Let $\cF = \cL_\sigma[-j]$ be a shifted simple sheaf on a fan $\Delta$, where $j\in \Z_{\geq 0}$.
Then the Hodge--Deligne polynomial of $(\ov{\cF(\Delta)}, F_\bullet, W^\bullet)$ is given by
\[
E(\ov{\cF(\Delta)}, F_\bullet, W^\bullet; u, v) = u^j v^{\dim (\sigma) - j} P(\ov{\cL_\sigma(\Delta)}; uv).
\]
\end{lem}

\begin{proof}
By Proposition~\ref{p:charofweight}
and \eqref{eq:SubmoduleSentToSubspace}, the weight filtration $W^\bullet$ on $\ov{\cF(\Delta)}$ is given by
\[
W^r
=
 \ov{\cF(\Delta)}^{\geq j+\frac{r-\dim(\sigma)}{2}}.
\]
Hence the associated graded vector space $\Gr_W^r \ov{\cF(\Delta)}$ is given by
\[
\Gr_W^r \ov{\cF(\Delta)} =
\begin{cases}
\ov{\cF(\Delta)}^{j + \frac{r-\dim(\sigma)}{2}} & \text{if } r - \dim(\sigma) \text{ is even}, \\
0 & \text{otherwise.}
\end{cases}
\]

We restrict to the case when $r-\dim(\sigma)$ is even.
Let $k = \frac{r-\dim(\sigma)}{2}$.
Then 
the associated graded vector space is given by
\[
\Gr^{F}_p \Gr_{W}^{r} \ov{\cF(\Delta)} =
\begin{cases}
\ov{\cF(\Delta)}^{j + k} & \text{if }
p = j+ k, \\
0 & \text{otherwise.}
\end{cases}
\]
Thus the Hodge--Deligne polynomial
is given by
\begin{align*}
E(\ov{\cF(\Delta)}, F_\bullet, W^\bullet; u, v)
& =
\sum_{p, r}
\dim_\R(
\Gr^F_p \Gr_{W}^{r}
\ov{\cF(\Delta)}
)
u^p v^{r-p}\\
& =
\sum_{k}
\dim_\R(
\ov{\cL_{\sigma}[-j](\Delta)}^{j+k}
)
u^{j+k} v^{\dim(\sigma) + k - j}\\
& =
u^{j} v^{\dim(\sigma)- j}
\sum_{k}
\dim_\R(
\ov{\cL_{\sigma}(\Delta)}^{k}
)
(u v)^{k}.
\end{align*}
The result then follows from the definition of $P(\ov{\cL_\sigma(\Delta)}; t)$ in Section~\ref{sec:puresheaves}.
\end{proof}

\subsection{The case of a subdivision}

\begin{defn}
Let $\pi\colon \Sigma\to \Delta$ be a fan subdivision and $\cF$ be a pure sheaf on $\Sigma$.
The \emph{monodromy weight filtration} $M^\bullet$ of $(\cF, \pi)$ is defined to be the decreasing filtration on $\ov{\pi_*\cF(\Delta)}$ given by
\[
M^r := \rho((\pi_*\cW^r\cF)(\Delta)),
\]
where $\cW^r \cF$ is the $r$-weight sheaf of $\cF$ on $\Sigma$ and $\rho\colon (\pi_*\cF)(\Delta)\to \ov{(\pi_*\cF)(\Delta)}$ is the quotient map.
\end{defn}

The monodromy weight filtration commutes with direct sums:
for every decomposition $\cF \simeq \oplus_\alpha \cL_{\sigma_\alpha}[-j_\alpha]$ of $\cF$ into shifted simple sheaves, the monodromy weight filtration of $(\cF, \pi)$ is the direct sum of the monodromy weight filtrations of the $(\cL_{\sigma_\alpha}[-j_\alpha], \pi)$'s.

Recall from Section~\ref{subsec:pureonquasi} that
the direct image sheaf $\pi_* \cF$ is a pure sheaf on $\Delta$.
Thus we can also consider 
the weight sheaves $\cW^r \pi_* \cF$ of the pure sheaf $\pi_*\cF$ on $\Delta$.
Hence the graded vector space $\ov{\pi_*\cF(\Delta)}$ is now equipped with 
the Hodge filtration $F_\bullet$,
the weight filtration $W^\bullet$
and the monodromy weight filtration $M^\bullet$.

\begin{defn}
Let $\pi\colon \Sigma\to \Delta$ be a fan subdivision and $\cF$ be a pure sheaf on $\Sigma$.
The \emph{limit Hodge--Deligne polynomial} of $(\ov{\pi_*\cF(\Delta)}, F_\bullet, M^\bullet)$ is defined to be
\[
E(\ov{\pi_*\cF(\Delta)}, F_\bullet, M^\bullet; u, v) =
\sum_{p, q} \dim_\R
(
\Gr^F_p \Gr_M^{p+q} \ov{\pi_*\cF(\Delta)}
) u^p v^q.
\]
\end{defn}

By specializing $u\mapsto t$ and $v\mapsto 1$, we recover the Poincar\'e polynomial
of
$\ov{\pi_*\cF(\Delta)}$.

\begin{defn}
Let $\pi\colon \Sigma\to \Delta$ be a fan subdivision and $\cF$ be a pure sheaf on $\Sigma$.
The \emph{refined limit Hodge--Deligne polynomial} of $(\ov{\pi_*\cF(\Delta)}, F_\bullet, W^\bullet, M^\bullet)$ is defined to be
\[
E(\ov{\pi_*\cF(\Delta)}, F_\bullet, W^\bullet, M^\bullet; u, v, w) =
\sum_{p, q, r} \dim_\R
(
\Gr^F_p \Gr_M^{p+q} \Gr_W^{r} \ov{\pi_*\cF(\Delta)}
) u^p v^q w^r.
\]
\end{defn}

By specializing $u\mapsto u w^{-1}$ and $v\mapsto 1$, we recover the Hodge--Deligne polynomial
of $(\ov{\pi_*\cF(\Delta)}, F_\bullet, W^\bullet)$
from the refined limit Hodge--Deligne polynomial.
By specializing $w\mapsto 1$, we recover the limit Hodge--Deligne polynomial
of $(\ov{\pi_*\cF(\Delta)}, F_\bullet, M^\bullet)$
from the refined limit Hodge--Deligne polynomial.

\begin{lem}\label{lem:mixedpolyslimit}
Let $\pi\colon \Sigma\to \Delta$ be a fan subdivision and $\cF = \cL_\sigma[-j]$ be a shifted simple sheaf on $\Sigma$, where $j\in \Z_{\geq 0}$.
Then we have the following:
\begin{enumerate}
\item \label{item:lem21}
The Hodge--Deligne polynomial of $(\ov{\pi_*\cF(\Delta)}, F_\bullet, W^\bullet)$ is given by
\[
E(\ov{\pi_*\cF(\Delta)}, F_\bullet , W^\bullet; u, v) =
\sum_{\tau\in \Delta}
u^j v^{\dim(\tau) - j}
L(\pi_* \cL_\sigma, \tau; u v^{-1}) \cdot
P(\ov{\cL_\tau(\Delta)}; uv).
\]

\item \label{item:lem22}
The limit Hodge--Deligne polynomial of $(\ov{\pi_*\cF(\Delta)}, F_\bullet, M^\bullet)$ is given by
\[
E(\ov{\pi_*\cF(\Delta)}, F_\bullet, M^\bullet; u, v) =  u^j v^{\dim( \sigma) - j} P(\ov{\pi_*\cL_\sigma(\Delta)}; u v).
\]

\item \label{item:lem23}
The refined limit Hodge--Deligne polynomial of $(\ov{\pi_*\cF(\Delta)}, F_\bullet, W^\bullet, M^\bullet)$ is given by
\[
E(\ov{\pi_*\cF(\Delta)}, F_\bullet, W^\bullet, M^\bullet ; u, v, w)
=
u^j
v^{\dim(\sigma)-j}
\sum_{\tau\in \Delta}
w^{\dim(\tau) }
L(\pi_*\cL_\sigma, \tau; uv) \cdot
P(\ov{\cL_\tau(\Delta)}; u v w^2).
\]
\end{enumerate}
\end{lem}

\begin{proof}
It suffices to prove only the third statement, 
since the other statements follow from specializations.

Since the direct image sheaf $\pi_* \cF$ is a pure sheaf on $\Delta$, we choose a decomposition of $\pi_* \cF$ into shifts of simple sheaves on $\Delta$:
\[
\pi_* \cF = \bigoplus_{\alpha} \cL_{\tau_{\alpha}}[-j-j_{\alpha}].
\]
We fix $\alpha$ and let $V = \ov{\cL_{\tau_{\alpha}}[-j-j_{\alpha}](\Delta)}$.
By Proposition~\ref{p:charofweight}
and \eqref{eq:SubmoduleSentToSubspace}, the weight filtration $W^\bullet$ on $V$,
induced by the weight sheaves of $\pi_* \cF$ on $\Delta$,
is given by
\[
W^r
=
 V^{\geq j+j_\alpha + \frac{r-\dim(\tau_\alpha)}{2}}.
\]
Hence the associated graded vector space $\Gr_W^r V$ is given by
\[
\Gr_W^r V =
\begin{cases}
V^{j + j_\alpha + \frac{r-\dim(\tau_\alpha)}{2}} & \text{if } r - \dim(\tau_\alpha) \text{ is even,}\\
0 & \text{otherwise.}
\end{cases}
\]

We then consider the induced filtration of ${M}^\bullet$ on $\Gr_W^r V$.
By Proposition~\ref{p:charofweight}
and \eqref{eq:SubmoduleSentToSubspace},
%since $\cF$ is a shifted simple sheaf on $\Sigma$,
the monodromy weight filtration $M^\bullet$ on $V$,
induced by the weight sheaves of $\cF$ on $\Sigma$,
is given by
\[
M^{p+q}
 = V^{\geq j+ \frac{p+q-\dim(\sigma)}{2}}.
\]
Hence the induced filtration of $M^\bullet$ on
$\Gr_W^r V$
is given by
\[
M^{p+q}
=
\begin{cases}
V^{j + j_\alpha + \frac{r-\dim(\tau_\alpha)}{2}}
& 
\begin{split}
& \text{if }  
 r-\dim(\tau_\alpha) \text{ is even and } \\
& \qquad p+q \leq 2 j_\alpha + r-\dim(\tau_\alpha) + \dim(\sigma)
,
\end{split}
\\
0 & \text{otherwise,}
\end{cases}
\]
and the associated graded vector space $\Gr_{M}^{p+q} \Gr_{W}^{r} V$ is given by
\[
\Gr_{M}^{p+q} \Gr_{W}^{r} V =
\begin{cases}
V^{j + j_\alpha + \frac{r-\dim(\tau_\alpha)}{2}} & 
\begin{split}
& \text{if }  
 r-\dim(\tau_\alpha) \text{ is even and } \\
& \qquad p+q  = 2 j_\alpha + r-\dim(\tau_\alpha) + \dim(\sigma)
,
\end{split}
\\
0 & \text{otherwise.}
\end{cases}
\]

We restrict to the case when $r - \dim(\tau_\alpha)$ is even and $p+q  = 2 j_\alpha + r-\dim(\tau_\alpha) + \dim(\sigma)$.
Let $k = \frac{r-\dim(\tau_\alpha)}{2}$.
Then 
the associated graded vector space is given by
\[
\Gr^{F}_p \Gr_{M}^{p+q} \Gr_{W}^{r} V =
\begin{cases}
V^{j + j_\alpha + k} & \text{if }
p = j+j_\alpha + k, \\
0 & \text{otherwise.}
\end{cases}
\]
Thus the refined limit Hodge--Deligne polynomial
is given by
\begin{align*}
& E(\ov{\pi_*\cF(\Delta)}, F_\bullet, W^\bullet, M^\bullet; u, v, w) \\
& \qquad =
\sum_{\alpha}
\sum_{p, q, r}
\dim_\R(
\Gr^F_p \Gr_{M}^{p+q} \Gr_{W}^{r}
\ov{\cL_{\tau_\alpha}[-j-j_\alpha](\Delta)}
)
u^p v^q w^r \\
& \qquad =
\sum_{\alpha}
\sum_{k}
\dim_\R
(
\ov{\cL_{\tau_\alpha}[-j-j_\alpha](\Delta)}^{
j+j_\alpha + k}
)
u^{j+j_\alpha+k}
v^{\dim(\sigma)+j_\alpha + k - j}
w^{\dim(\tau_\alpha) + 2k}.
\end{align*}
By removing the shift,
we get
\begin{align*}
& E(\ov{\pi_*\cF(\Delta)}, F_\bullet, W^\bullet, M^\bullet; u, v, w) \\
& \qquad = 
\sum_{\alpha}
\sum_{k}
\dim_\R
(
\ov{\cL_{\tau_\alpha}(\Delta)}^{
k}
)
u^{j+j_\alpha+k}
v^{\dim(\sigma)+j_\alpha + k - j}
w^{\dim(\tau_\alpha) + 2k}\\
& \qquad =
u^j
v^{\dim(\sigma)-j}
\sum_{\tau \in \Delta}
w^{\dim(\tau) }
\sum_{\substack{\alpha\\ \tau_\alpha = \tau}}
u^{j_\alpha } v^{j_\alpha}
\sum_k \dim_\R
(\ov{\cL_{\tau}(\Delta)}^{ k})
u^k v^k w^{2k} \\
& \qquad =
u^j
v^{\dim(\sigma)-j}
\sum_{\tau \in \Delta}
w^{\dim(\tau) }
L(\pi_*\cL_\sigma, \tau; uv) \cdot
P(\ov{\cL_\tau(\Delta)}; u v w^2),
\end{align*}
where the last equality follows from the definitions of
$P(\ov{\cL_\tau(\Delta)}; t)$
and
$L(\pi_*\cL_\sigma, \tau; t)$ in Section~\ref{sec:puresheaves}.
\end{proof}

\section{The mixed \texorpdfstring{$h$}{h}-polynomial from the minimal extension sheaf}\label{sec:hpoly}

In this section, 
we first introduce the toric $h$-polynomial
following \cite[Section~3.16]{Sta12}.
We then introduce the mixed $h$-polynomial, 
an enrichment of the toric $h$-polynomial, 
following \cite[Section~5]{KS16a}.
At the end, we show that the Hodge--Deligne polynomial defined in Section~\ref{sec:polysfromfiltrations} is equal to the mixed $h$-polynomial.

\subsection{The toric \texorpdfstring{$h$}{h}-polynomial}

For each cone $\sigma$, 
we define the \emph{toric $g$-polynomial} $g(\langle \sigma \rangle; t)$ \cite[Section~2]{Sta87} recursively as follows:
Let $g(\langle o \rangle; t) : = 1$.
Then for a cone $\sigma$ of positive dimension, 
we define $g(\langle \sigma \rangle; x)$ to be the unique polynomial of degree strictly less than $\dim(\sigma)/2$ satisfying
\[
t^{\dim(\sigma)} g(\langle \sigma \rangle; t^{-1})=
 \sum_{o\leq \tau \leq \sigma} g(\langle \tau \rangle; t)  \cdot
(t-1)^{\dim(\sigma)-\dim(\tau)}.
\]

\begin{defn}[{\cite[Section~2]{Sta87}}]
For a quasi-convex fan $\Delta$, its \emph{toric $h$-polynomial} $h(\Delta; t)$
is defined to be the unique polynomial satisfying 
\[
t^{\dim (\Delta)} h(\Delta; t^{-1}) = \sum_{\sigma \in \Delta} g(\langle\sigma \rangle; t)  \cdot
(t-1)^{\dim (\Delta) -\dim (\sigma)}.
\]
\end{defn}

Both the toric $g$-polynomial and the toric $h$-polynomial depend on the poset structure of the fan only.
%They are motivated by the intersection cohomology of toric varieties \cite[Section~2]{Sta87}.
When $\sigma$ is a cone, 
by comparing their definitions, we have that $h(\langle \sigma \rangle; t) = g(\langle\sigma \rangle; t)$ \cite[Example~3.14]{KS16a}.
Thus we may use them interchangeably.

\begin{example}[{\cite[Corollary~2.2]{Sta87}}]
The toric $h$-polynomial of a simplicial quasi-convex fan $\Delta$
is easy to describe:
it satisfies
\[
t^{\dim(\Delta)} h(\Delta; t^{-1}) = \sum_{i=0}^{\dim(\Delta)} f_i (t-1)^{\dim(\Delta) - i},
\]
where $f_i$ is the number of $i$-dimensional faces in $\Delta$.
\end{example}

Given a proper fan morphism $\pi\colon \Sigma\to \langle \sigma \rangle$,
its \emph{local $h$-polynomial} $\ell^h_{\langle \sigma \rangle}(\Sigma; t)$ is defined to be \cite[Definition~4.1]{KS16a}
\[
\ell^h_{\langle\sigma \rangle}(\Sigma; t) :=
\sum_{\tau \in \langle \sigma \rangle}
h(\pi^{-1}(\langle \tau\rangle ); t) \cdot (-1)^{\dim(\sigma)-\dim (\tau)} g((\link_{\langle \sigma \rangle} \tau)^*; t),
\]
where
$(\link_{\langle \sigma \rangle} \tau)^*$ is the \textit{dual poset} of the face poset of $\link_{\langle \sigma \rangle} \tau$.
Here,
the preimage fan $\pi^{-1}(\langle \tau \rangle)$ has support isomorphic to $\tau \times \R^{\dim(\Sigma) - \dim(\sigma)}$,
hence it is quasi-convex in $\Span(\pi^{-1}(\langle \tau \rangle))$
and its toric $h$-polynomial is well-defined. 
We note that the local $h$-polynomial is of degree at most $\dim(\Sigma)$ \cite[Corollary~4.5]{KS16a}.
This polynomial was introduced by Stanley \cite[Section~2]{Sta92}; see \cite{Ath16} for a survey. 

\begin{remark}
A proper fan morphism is a \textit{strong formal subdivision} as defined in \cite[Definition~3.17]{KS16a}.
Hence we can use the results in \cite{KS16a}.
\end{remark}

There is a decomposition theorem for the toric $h$-polynomial.
It was proved by Stanley \cite[Theorem~3.3]{Sta92} for simplicial subdivisions of a simplex, by Athanasiadis \cite[Theorem~3.3]{Ath12} for homology subdivisions of a simplicial complex and by Katz and Stapledon \cite[Remark~4.2]{KS16a} for strong formal subdivisions of an Eulerian poset.
We prove the theorem for proper fan morphisms of quasi-convex fans below, 
making use of the case of the theorem proved by Katz and Stapledon.

\begin{prop}\label{p:hPolyGeneralDecompTheorem}
For a proper fan morphism $\pi\colon \Sigma\to \Delta$ of quasi-convex fans, we have
\begin{equation}\label{eq:decompforh}
h(\Sigma; t) = \sum_{\sigma \in \Delta} \ell^h_{\langle \sigma\rangle}(\pi^{-1}(\langle \sigma\rangle); t)  \cdot h(\link_\Delta \sigma ; t).
\end{equation}
\end{prop}

\begin{proof}
By definition,
the toric $h$-polynomial of $\Sigma$ satisfies
\begin{align*}
t^{\dim(\Sigma)}h(\Sigma; t^{-1}) 
&= 
\sum_{\tau\in \Sigma}
g(\langle \tau \rangle; t)
\cdot (t-1)^{\dim(\Sigma) - \dim(\tau)}\\
&= 
\sum_{\gamma\in \Delta}
\sum_{\substack{\tau\in \Sigma\\ \pi(\tau) = \gamma}}
g(\langle \tau \rangle; t)
\cdot 
(t-1)^{\dim(\Sigma) - \dim(\tau)}.
\end{align*}
%where we group the cones in $\Sigma$ based on their images in $\Delta$.
By \cite[Proposition~3.29(3)]{KS16a},
we have
\[
h(\pi^{-1}(\langle \gamma \rangle); t)
=
\sum_{\substack{\tau\in \Sigma\\ \pi(\tau) = \gamma}}
g(\langle \tau \rangle; t)
\cdot
 (t-1)^{\dim(\Sigma) - \dim(\Delta) + \dim(\gamma) - \dim(\tau)}.
\]
Thus we have
\begin{align*}
t^{\dim(\Sigma)}h(\Sigma; t^{-1}) 
& = 
\sum_{\gamma\in \Delta}
h(\pi^{-1}(\langle \gamma \rangle); t)
\cdot 
(t - 1)^{\dim(\Delta) -\dim(\gamma) }\\
& = 
\sum_{\gamma\in \Delta}
\bigg(
\sum_{\sigma \in \langle \gamma \rangle}
\ell^h_{\langle \sigma\rangle}(\pi^{-1}(\langle \sigma\rangle); t) 
\cdot 
h(\lk_{\langle \gamma \rangle} \sigma; t)
\bigg)
\cdot 
(t - 1)^{\dim(\Delta) -\dim(\gamma) },
\end{align*}
where we use 
the case of the theorem for the fan subdivision of $\langle \tau \rangle$
as proved in \cite[Remark~4.2]{KS16a}.
Since 
$\lk_{\langle \gamma \rangle} \sigma$
is equal to the fan $\langle \ov{\gamma} \rangle$ in $\lk_\Delta \sigma$, 
we have 
$h(\lk_{\langle \gamma \rangle} \sigma; t)=
h(\langle \ov{\gamma} \rangle; t) = 
g(\langle \ov{\gamma} \rangle; t)$.
Hence by switching the order of summation 
we have
\begin{align*}
t^{\dim(\Sigma)}h(\Sigma; t^{-1}) 
& = 
\sum_{\sigma \in \Delta}
\sum_{\ov{\gamma} \in \lk_\Delta \sigma}
\ell^h_{\langle \sigma\rangle}(\pi^{-1}(\langle \sigma\rangle); t) 
\cdot 
g(\langle \ov{\gamma} \rangle; t)
\cdot 
(t - 1)^{\dim(\lk_\Delta \sigma) -\dim(\ov{\gamma}) }
\\
& = 
\sum_{\sigma \in \Delta}
\ell^h_{\langle \sigma\rangle}(\pi^{-1}(\langle \sigma\rangle); t) 
\cdot 
\bigg(
\sum_{\ov{\gamma} \in \lk_\Delta \sigma}
g(\langle \ov{\gamma} \rangle; t)
\cdot 
(t - 1)^{\dim(\lk_\Delta \sigma) -\dim(\ov{\gamma}) }
\bigg)
.
\end{align*}
Now, 
by definition, 
the toric $h$-polynomial of the fan $\lk_\Delta \sigma$
satisfies
\[
t^{\dim(\Delta) - \dim(\sigma)}h(\lk_\Delta \sigma; t^{-1})
=
\sum_{\ov{\gamma} \in \lk_\Delta \sigma}
g(\langle \ov{\gamma} \rangle; t)
\cdot 
(t - 1)^{\dim(\lk_\Delta \sigma) -\dim(\ov{\gamma}) }
.
\]
Furthermore,
the local $h$-polynomial is symmetric \cite[Corollary~4.5]{KS16a}:
\begin{align*}
\ell^h_{\langle \sigma\rangle}(\pi^{-1}(\langle \sigma\rangle); t)
&=
t^{\dim(\Sigma) - \dim(\Delta) + \dim(\sigma)}
\ell^h_{\langle \sigma\rangle}(\pi^{-1}(\langle \sigma\rangle); t^{-1}).
\end{align*}
Thus we have
\[
t^{\dim(\Sigma)}h(\Sigma; t^{-1}) 
=
t^{\dim(\Sigma)}
\sum_{\sigma \in \Delta}
\ell^h_{\langle \sigma\rangle}(\pi^{-1}(\langle \sigma\rangle); t^{-1})
\cdot 
h(\lk_\Delta \sigma; t^{-1}).
\]
The result then follows from 
canceling 
$t^{\dim(\Sigma)}$
and
replacing $t^{-1}$ by $t$.
\end{proof}

Using the decomposition theorem for the toric $h$-polynomial, 
we define the following multivariable polynomial.

\begin{defn}[{\cite[Definition~5.1]{KS16a}}]\label{defn:mixedh}
Let $\pi\colon \Sigma\to \Delta$ be a proper fan morphism of quasi-convex fans.
We define the \emph{mixed $h$-polynomial} $h_\Delta(\Sigma; u, v) \in \R[u, v]$ to be
\[
h_\Delta(\Sigma; u, v)
=
\sum_{\sigma\in \Delta}
v^{
\dim(\pi^{-1}(\langle \sigma\rangle ))}
\ell^h_{\langle \sigma\rangle}(\pi^{-1}(\langle \sigma\rangle); u v^{-1}) \cdot h(\link_\Delta \sigma ; uv).
\]
\end{defn}

By specializing $v\mapsto 1$, we recover the toric $h$-polynomial of $\Sigma$.

\subsection{The minimal extension sheaf}\label{sec:minextsheaf}

We first prove a lemma relating the simple sheaf based at $\sigma$ and the minimal extension sheaf on $\lk_\Delta \sigma$.
While the lemma is known to the community, 
for the purpose of completeness we provide a proof here. 

\begin{lem}
\label{lem:SimpleSheafAndMinimalExtensionSheaf}
Let $\Delta$ be a fan in $\R^d$ and $\sigma \in \Delta$ be a cone.
Then we have
\[
P(\ov{\cL_\sigma(\Delta)}; t) 
=
P(\ov{\cL_{\lk_\Delta \sigma}(\lk_\Delta \sigma)}; t).
\]
\end{lem}

\begin{proof}
Let $i\colon \lk_\Delta\sigma \to \Delta$ be the injective map that sends $\ov{\tau}$ to $\tau$.
This is a continuous map.
Thus the direct image sheaf $i_* \cL_{\lk_\Delta\sigma}$ 
is a sheaf of 
$\Z_{\geq 0}$-graded
$i_* \cA_{\lk_\Delta\sigma}$-modules
on $\Delta$.

Recall that $A_\sigma$ is the $\Z_{\geq 0}$-graded ring of polynomial functions on $\Span(\sigma)$.
Thus 
$A_{\sigma}\otimes_\R i_* \cL_{\lk_\Delta\sigma}$ 
is 
a $\Z_{\geq 0}$-graded sheaf on $\Delta$.
While this is a sheaf of 
$A_{\sigma}
\otimes_\R 
i_* \cA_{\lk_\Delta\sigma}$-modules,
we can consider it as a sheaf of $\cA_{\Delta}$-modules:
For a cone $\tau$,
the stalk 
$(A_{\sigma}\otimes_\R i_* \cL_{\lk_\Delta\sigma})(\tau)$ is nonzero precisely 
when $\tau \geq \sigma$.
For such a cone $\tau$, we have
\[
(A_{\sigma} \otimes_\R 
i_* \cA_{\lk_\Delta\sigma})(\tau)
=
A_\sigma \otimes_\R A_{\ov{\tau}} \simeq A_\tau
=\cA_{\Delta}(\tau),
\]
where 
$A_{\ov{\tau}}$ is the graded ring of polynomial functions on $\Span(\tau)/\Span(\sigma)$.

We prove that 
$A_{\sigma}\otimes_\R i_* \cL_{\lk_\Delta\sigma}$ is the simple sheaf based at $\sigma$
by 
checking the conditions in Definition~\ref{defn:SimpleSheaves}:
\begin{itemize}
\item At $\sigma$, 
we have 
$(A_{\sigma}\otimes_\R i_* \cL_{\lk_\Delta\sigma})(\sigma) = A_{\sigma}\otimes_\R \R \simeq A_\sigma$.

\item For each $\tau \geq \sigma$, 
since $\cL_{\lk_\Delta\sigma}(\ov{\tau})$ is a free 
$A_{\ov{\tau}}$-module,
the stalk
$(A_{\sigma}\otimes_\R i_* \cL_{\lk_\Delta\sigma})(\tau)$ 
is a free $A_\sigma\otimes_\R A_{\ov{\tau}} \simeq A_\tau$-module.

\item 
The restriction map of $A_{\sigma}\otimes_\R i_* \cL_{\lk_\Delta\sigma}$ 
is the tensor product of the identity map on $A_{\sigma}$ and the restriction map 
of $i_*\cL_{\lk_\Delta\sigma}$.
Given a cone $\tau > \sigma$,
since $A_\tau \simeq A_\sigma\otimes_\R A_{\ov{\tau}}$,
the map induced by the restriction map of $A_{\sigma}\otimes_\R i_* \cL_{\lk_\Delta\sigma}$
is given by 
\[
\Big(
\ov{\id}\colon \ov{A_\sigma} \to \ov{A_\sigma} 
\Big)
\otimes_\R 
\Big(
\ov{\res}\colon 
\ov{\cL_{\lk_\Delta \sigma}(\ov{\tau})}
\to 
\ov{\cL_{\lk_\Delta \sigma}(\partial \ov{\tau})}
\Big).
\]
Since both $\ov{\id}$ and $\ov{\res}$ are isomorphisms of graded vector spaces,
their tensor product is also an isomorphism of graded vector spaces. 
\end{itemize}

As a result, 
we have
$\cL_\sigma(\Delta) \simeq (A_{\sigma}\otimes_\R i_* \cL_{\lk_\Delta\sigma})(\Delta)$.
Now, 
note that 
$A_\Delta \simeq A_\sigma \otimes_\R A_{\lk_\Delta \sigma}$,
where $A_\Delta$ and $A_{\lk_\Delta \sigma}$ are the rings of global polynomial functions on $\R^d$ and $\R^d/\Span(\sigma)$, respectively.
Hence we have 
\[
\ov{
(A_{\sigma}\otimes_\R i_* \cL_{\lk_\Delta\sigma})(\Delta)} 
\simeq 
\R \otimes_\R \ov{i_* \cL_{\lk_\Delta\sigma}(\Delta)}
\simeq 
\ov{\cL_{\lk_\Delta\sigma}(\lk_\Delta\sigma)},
\]
and the result follows.
\end{proof}

Given a fan subdivision $\pi\colon \Sigma\to \Delta$ and a cone $\sigma$ in $\Sigma$,
we naturally have a map $\pi_\sigma\colon \lk_\Sigma \sigma\to \lk_\Delta \pi(\sigma)$ that sends $\ov{\gamma}$ to $\ov{\pi(\gamma)}$.
Furthermore, this map
is a proper fan morphism
induced by the linear map $P_\sigma \colon \R^d / \Span(\sigma) \to \R^d / \Span(\pi(\sigma))$
\cite[Exercise~3.4.10]{CLS11}.

\begin{prop}\label{p:hpolyPoincar\'e}
The following are true:
\begin{enumerate}
\item \label{item:statementonehpolypoincare}
For a cone $\sigma$ in a quasi-convex fan $\Delta$, we have 
$
P(\ov{\cL_\sigma(\Delta)}; t) = h(\lk_\Delta \sigma; t)$.
In particular, we have 
$
P(\ov{\cL_\Delta(\Delta)}; t) = h(\Delta; t)$.

\item \label{item:localhPoincar\'e}
For a fan subdivision $\pi\colon \Sigma\to \langle \tau \rangle$, where $\tau$ is a cone, and a cone $\sigma \in \Sigma$, we have
\[
L(\pi_* \cL_\sigma, \tau; t) =
\ell^h_{\lk_{\langle \tau \rangle} \pi(\sigma)}
(\lk_{\Sigma} \sigma; t),
\]
the local $h$-polynomial of the proper fan morphism 
$\pi_\sigma \colon \lk_{\Sigma}  \sigma \to \lk_{\langle \tau \rangle} \pi(\sigma)$.
\end{enumerate}
\end{prop}

\begin{proof}
Statement~\eqref{item:statementonehpolypoincare} 
for $\sigma = o$ is a known consequence of the hard Lefschetz theorem (Theorem~\ref{t:hardlef}) of Karu \cite[Section~0]{Kar04} and \cite[Theorem~5.3]{BBFK02}.
For a general simple sheaf $\cL_\sigma$ based at a cone $\sigma$,
by Lemma~\ref{lem:SimpleSheafAndMinimalExtensionSheaf} we have
\[
P(\ov{\cL_\sigma(\Delta)}; t) 
=
P(\ov{\cL_{\lk_\Delta \sigma}(\lk_\Delta \sigma)}; t).
\]
Hence the result follows from applying the case when $\sigma = o$ to the fan $\lk_\Delta \sigma$,
which is quasi-convex
by 
Proposition~\ref{p:linkofqcisqc}.

While we believe Statement~\eqref{item:localhPoincar\'e}
is known to the community (see for example \cite[p.~861]{Kar19}), 
for completeness purpose we give a proof here.
We first apply \eqref{eq:Poincar\'erecursion} to the direct image sheaf $\pi_* \cL_\sigma$ to get 
\begin{align*}
P(\ov{\cL_\sigma(\Sigma)}; t) = P(\ov{\pi_* \cL_\sigma(\langle \tau \rangle)}; t)
 & = \sum_{\gamma \in \langle \tau \rangle} L(\pi_* \cL_\sigma, \gamma; t)
  \cdot
  P(\ov{\cL_\gamma(\langle \tau \rangle)}; t)\\
& = \sum_{\substack{\gamma\in \langle \tau \rangle \\ \gamma \geq \pi(\sigma)}}
L(\pi_* \cL_\sigma, \gamma; t)
 \cdot
 P(\ov{\cL_\gamma(\langle \tau \rangle)}; t),
\end{align*}
since $\pi_* \cL_\sigma(\gamma) = 0$ unless $\gamma\geq \pi(\sigma)$.
On the other hand,
the map $\pi_\sigma \colon \lk_\Sigma \sigma \to \lk_{\langle \tau \rangle} \pi(\sigma)$ is a proper fan morphism of quasi-convex fans in their respective linear spans.
Thus \eqref{eq:decompforh} implies that 
\begin{align*}
h(\lk_\Sigma \sigma;t ) & = \sum_{\ov{\gamma} \in \lk_{\langle \tau \rangle} \pi(\sigma)}
\ell^h_{\langle \ov{\gamma}\rangle}(\pi_\sigma^{-1}({\langle \ov{\gamma}\rangle}); t)
 \cdot
h(\lk_{\lk_{\langle \tau \rangle}\pi(\sigma)} \ov{\gamma} ; t)\\
& = \sum_{\substack{\gamma\in \langle \tau \rangle \\ \gamma \geq \pi(\sigma)}}
\ell^h_{\lk_{\langle \gamma \rangle} \pi(\sigma)}(\lk_{\pi^{-1}(\langle \gamma \rangle)} \sigma; t)
 \cdot
 h(\lk_{\langle \tau \rangle} \gamma ; t),
\end{align*}
where we replace the fans in 
$\lk_\Sigma \sigma$
and 
$\lk_{\langle \tau \rangle} \pi(\sigma)$
with equal fans in $\Sigma$ and $\langle \tau \rangle$.
By an induction on $\dim(\sigma)$, 
we may assume 
$L(\pi_* \cL_\sigma, \gamma; t) =
\ell^h_{\lk_{\langle \gamma\rangle} \pi(\sigma)}
(\lk_{\pi^{-1}(\langle \gamma \rangle)} \sigma; t)$ when $\dim(\gamma) < \dim(\tau)$.
Now, 
since $P(\ov{\cL_\tau(\langle \tau\rangle)}; t) =  1 = 
h(\lk_{\langle \tau \rangle} \tau; t)$, 
we have
\begin{align*}
L(\pi_* \cL_\sigma, \tau; t) & = 
P(\ov{\cL_\sigma(\Sigma)}; t)
 -
 \sum_{\pi(\sigma) \leq \gamma < \tau}
L(\pi_* \cL_\sigma, \gamma; t)
 \cdot
 P(\ov{\cL_\gamma(\langle \tau \rangle)}; t)
 \\
 & = 
 h(\lk_\Sigma \sigma; t)
- 
\sum_{\pi(\sigma) \leq \gamma < \tau}
\ell^h_{\lk_{\langle \gamma \rangle} \pi(\sigma)}
(\lk_{\pi^{-1}(\langle \gamma\rangle)} \sigma; t)
 \cdot
 h(\lk_{\langle \tau \rangle} \gamma; t)\\
 & = 
 \ell^h_{\lk_{\langle \tau \rangle} \pi(\sigma)}
(\lk_{\Sigma} \sigma; t),
\end{align*}
where the second equality uses Statement~\eqref{item:statementonehpolypoincare}.
\end{proof}

Recall from Section~\ref{sec:polysfromfiltrations} the definition of the Hodge--Deligne polynomial.

\resmixedh

\begin{proof}
By statement~\eqref{item:lem21} in Lemma~\ref{lem:mixedpolyslimit}, the Hodge--Deligne polynomial of $\pi_* \cL_\Sigma$ is given by
\begin{align*}
E(\ov{\pi_* \cL_\Sigma(\Delta)}, F_\bullet, W^\bullet; u, v)
& = \sum_{\tau \in \Delta} v^{\dim(\tau)} L(\pi_* \cL_\Sigma, \tau; u v^{-1})
 \cdot
P(\ov{\cL_\tau(\Delta)}; uv)\\
& =\sum_{\tau \in \Delta} v^{\dim(\pi^{-1}(\langle\tau\rangle))}
\ell^h_{\langle \tau \rangle}(\pi^{-1}(\langle \tau \rangle) ; u v^{-1})
 \cdot
h(\lk_\Delta \tau; uv),
\end{align*}
where the last equality follows from Proposition~\ref{p:hpolyPoincar\'e}.
\end{proof}

\begin{remark}
The theorem can be generalized to a proper fan morphism, 
by studying the corresponding subdivision of non-pointed fans following \cite[p.~861]{Kar19}.
\end{remark}

The following results generalize 
\cite[Corollary~6.7]{KS16a} and 
\cite[Theorem~6.1]{KS16a}.

\begin{cor}\label{c:localhnonneg}
For a fan subdivision $\pi\colon \Sigma\to \langle\sigma \rangle$,
where $\sigma$ is a cone,
the local $h$-polynomial $\ell^h_{\langle \sigma \rangle}(\Sigma; t)$ has nonnegative coefficients.
If we further assume that $\pi$ is projective, 
then the coefficients of the local $h$-polynomial are
symmetric about degree $\dim(\sigma)/2$ and are
unimodal.
\end{cor}

\begin{proof}
By Proposition~\ref{p:hpolyPoincar\'e}
the local $h$-polynomial is equal to 
the local Poincar\'e polynomial,
hence the nonnegativity follows immediately. 
If the fan morphism is projective, 
then by 
the relative hard Lefschetz theorem
(Theorem~\ref{t:relhardlef}),
the coefficients of the local Poincar\'e polynomial 
$L(\pi_*\cL_\Sigma, \sigma; t)$
are symmetric about degree $\dim(\sigma)/2$ and are unimodal.
\end{proof}

\begin{cor}\label{c:mixedhproperties}
Let $\pi\colon \Sigma\to \Delta$ be a fan subdivision between quasi-convex fans.
Then the mixed $h$-polynomial $h_\Delta(\Sigma; u, v)$ has nonnegative coefficients.
If we further suppose that $\pi$ is projective and write $h_\Delta(\Sigma; u, v) = \sum_{i, j} h_{i, j} u^i v^j$, 
then for each $k\in \Z_{\geq 0}$
the sequence $\{h_{i, k-i}\}_{i=0, \dots, k}$ is symmetric and unimodal.
\end{cor}

\begin{proof}
The nonnegativity immediately follows from Theorem~\ref{t:mixedh}.

Suppose further that $\pi$ is projective.
%By the definition of the mixed $h$-polynomial,
It suffices to prove that 
for each $k\in \Z_{\geq 0}$,
the coefficients of the monomials of total degree $k$
in each 
$v^{\dim(\sigma)}
\ell^h_{\langle \sigma \rangle}(\pi^{-1}(\langle \sigma \rangle); u v^{-1}) \cdot h(\lk_\Delta \sigma; uv)$
form a nonnegative, symmetric and unimodal sequence.
By Proposition~\ref{p:hpolyPoincar\'e},
the toric $h$-polynomial $h(\lk_\Delta \sigma; uv)$ has nonnegative coefficients.
Since its input is $uv$, 
its coefficients are symmetric and unimodal
 at each fixed total degree.
Hence it remains to prove that $v^{\dim(\sigma)}
\ell^h_{\langle \sigma \rangle}(\pi^{-1}(\langle \sigma \rangle); u v^{-1})$
has nonnegative, symmetric and unimodal coefficients,
but that is precisely Corollary~\ref{c:localhnonneg}.
\end{proof}

\section{Multivariable polynomials from the Ehrhart sheaf}\label{sec:Ehrhart}

In this section, 
we first introduce the $h^*$-polynomial following \cite[Chapter~3]{BR15}.
We then introduce three multivariable enrichments of the $h^*$-polynomial following \cite[Sections~7--9]{KS16a}.
At the end, we define the \textit{Ehrhart sheaf} following \cite[Section~3]{Kar08}, and show that the filtrations defined in Section~\ref{sec:polysfromfiltrations} give these multivariable polynomials.

\subsection{Ehrhart theory}

A \textit{rational} cone $\sigma$ is a pointed cone in $\R^d$ that is generated by 
rational vectors.
The \textit{primitive lattice generators} of a rational cone are the first nonzero lattice points on its rays.

\begin{defn}[{\cite[Section~3]{Kar08}}]
A rational cone $\sigma$ is \emph{Gorenstein} if there exists a linear function $G_\sigma\colon \sigma \to \R$ such that 
$G_\sigma(\Z^d \cap \sigma) \subseteq \Z_{\geq 0}$ and $G_\sigma(v) = 1$ 
whenever $v$ is a primitive lattice generator of $\sigma$.
\end{defn}

If such a map $G_\sigma$ exists, 
then it is necessarily unique.

\begin{example}
One way of obtaining a Gorenstein cone is 
to take the cone over a $(d-1)$-dimensional lattice polytope in $\{x_d = 1\} \subseteq \R^d$.
Then 
the map 
$G_\sigma$ is given by 
$(x_1, \dots, x_d)\mapsto x_d$.
\end{example}

\begin{defn}
A fan $\Delta$ is \emph{Gorenstein} if each of its cones is Gorenstein.
\end{defn}

For a Gorenstein fan, 
by pasting the linear functions $G_\sigma$ of the cones together, 
we obtain a conewise linear function $G\colon |\Delta| \to \R$,
which we call the \emph{degree map}.
The degree map takes integral value at each lattice point in $|\Delta|$ and 
sends the primitive lattice generator of each ray to $1$.

We proceed to define the $h^*$-polynomial.
Given a Gorenstein fan $\Delta$ in $\R^d$, 
its \textit{Ehrhart series} is 
\[
\Ehr_\Delta(t) := \sum_{\lambda\in \Z^d \cap |\Delta|} t^{G(\lambda)}.
\]
The Ehrhart series satisfies \cite[Example~7.13]{Sta92}
\[
\Ehr_\Delta(t) = \frac{h^*(\Delta; t)}{(1-t)^{
\dim(\Delta)}},
\]
for a unique polynomial $h^*(\Delta; t)$ called the \emph{$h^*$-polynomial}.

Given a Gorenstein cone $\sigma$,
the \emph{local $h^*$-polynomial} $\ell^*(\langle \sigma\rangle ; t)$
is defined to be \cite[Definition~7.2]{KS16a}
\[
\ell^*(\langle \sigma\rangle ; t) = \sum_{\tau \in \langle \sigma\rangle} h^*(\langle\tau\rangle; t)
\cdot
(-1)^{\dim(\sigma) - \dim(\tau)}
g((\lk_{\langle \sigma\rangle} \tau)^*; t),
\]
where $(\lk_{\langle \sigma\rangle} \tau)^*$ is the \textit{dual poset} to the facet poset of $\lk_{\langle \sigma\rangle} \tau$.
The local $h^*$-polynomial was introduced by Stanley \cite[Example~7.13]{Sta92}
and
is motivated by an inclusion-exclusion formula weighted by the toric $g$-polynomial.

There is a decomposition theorem for the $h^*$-polynomial.
It was proved by Stanley \cite[Example~7.13]{Sta92} for the case when $\Delta$ is generated by a single simplicial cone and
Katz and Stapledon in \cite[Lemma~7.12]{KS16a} for the case when $\Delta$ is supported on a cone.
We prove the theorem for quasi-convex Gorenstein fans below,
making use of the case proved by Katz and Stapledon.

\begin{prop}
For a quasi-convex Gorenstein fan $\Delta$ with degree map $G$, we have
\begin{equation}\label{eq:hstardecomp}
h^*(\Delta; t)  =\sum_{\sigma\in \Delta} \ell^*(\langle \sigma \rangle; t)
\cdot
h(\lk_\Delta \sigma; t).
\end{equation}
\end{prop}

\begin{proof}
Given a cone $\tau \in \Delta$, 
we write $\tau^\circ$ for its relative interior
and let
\[
\Ehr_{\tau^\circ}(t) := 
\sum_{\lambda\in \Z^d \cap \tau^\circ} t^{G(\lambda)}.
\]
Since every lattice point in $|\Delta|$ is contained in the relative interior of exactly one cone in $\Delta$,
the Ehrhart series of $\Delta$ satisfies
\begin{align*}
\Ehr_\Delta(t^{-1}) 
& = 
\sum_{\tau\in \Delta} 
\Ehr_{\tau^\circ}(t^{-1})
\\
& = 
\sum_{\tau\in \Delta} 
(-1)^{\dim(\tau)}
\Ehr_{\langle \tau \rangle}(t),
\end{align*}
where 
the last equality is 
\textit{Ehrhart reciprocity} \cite[Proposition~4.1]{Mac71}.
Thus we have
\[
\frac{h^*(\Delta; t^{-1})}{(1-t^{-1})^{\dim(\Delta)}}
=
\sum_{\tau\in \Delta} 
(-1)^{\dim(\tau)}
\frac{h^*(\langle \tau \rangle; t)}{(1-t)^{\dim(\tau)}},
\]
or 
\begin{equation}
t^{\dim(\Delta)} 
h^*(\Delta; t^{-1})
=
\sum_{\tau\in \Delta} 
h^*(\langle \tau \rangle; t)
\cdot
(t-1)^{\dim(\Delta) - \dim(\tau)}.
\end{equation}
From this point onward, 
the proof is similar to that of 
Proposition~\ref{p:hPolyGeneralDecompTheorem}:
Using 
the theorem 
for the fan $\langle \tau \rangle$ 
as proved in \cite[Lemma~7.12]{KS16a},
we have
\begin{align*}
t^{\dim(\Delta)}h^*(\Delta; t^{-1}) 
& = 
\sum_{\tau\in \Delta}
h^*(\langle \tau \rangle; t)
\cdot 
(t - 1)^{\dim(\Delta) -\dim(\tau) }\\
& = 
\sum_{\tau\in \Delta}
\bigg(
\sum_{\sigma \in \langle \tau \rangle}
\ell^*(\langle \sigma \rangle; t) 
\cdot 
h(\lk_{\langle \tau \rangle} \sigma; t)
\bigg)
\cdot 
(t - 1)^{\dim(\Delta) -\dim(\tau) }.
\end{align*}
Since 
$\lk_{\langle \tau \rangle} \sigma$
is equal to the fan $\langle \ov{\tau} \rangle$ in $\lk_\Delta \sigma$, 
we have 
$h(\lk_{\langle \tau \rangle} \sigma; t)=
h(\langle \ov{\tau} \rangle; t) = 
g(\langle \ov{\tau} \rangle; t)$.
Hence by switching the order of summation 
we have
\begin{align*}
t^{\dim(\Delta)} h^*(\Delta; t^{-1}) 
& = 
\sum_{\sigma \in \Delta}
\sum_{\ov{\tau} \in \lk_\Delta \sigma}
\ell^*(\langle \sigma\rangle; t) 
\cdot 
g(\langle \ov{\tau} \rangle; t)
\cdot 
(t - 1)^{\dim(\lk_\Delta \sigma) -\dim(\ov{\tau}) }
\\
& = 
\sum_{\sigma \in \Delta}
\ell^*(\langle \sigma\rangle; t) 
\cdot 
\bigg(
\sum_{\ov{\tau} \in \lk_\Delta \sigma}
g(\langle \ov{\tau} \rangle; t)
\cdot 
(t - 1)^{\dim(\lk_\Delta \sigma) -\dim(\ov{\tau}) }
\bigg)
.
\end{align*}
Now, 
by definition, 
the toric $h$-polynomial of the fan $\lk_\Delta \sigma$
satisfies
\[
t^{\dim(\Delta) - \dim(\sigma)}h(\lk_\Delta \sigma; t^{-1})
=
\sum_{\ov{\tau} \in \lk_\Delta \sigma}
g(\langle \ov{\tau} \rangle; t)
\cdot 
(t - 1)^{\dim(\lk_\Delta \sigma) -\dim(\ov{\tau}) }
.
\]
Furthermore,
the local $h^*$-polynomial is symmetric \cite[Lemma~7.3]{KS16a}:
\begin{align*}
\ell^*(\langle \sigma\rangle; t) 
&=
t^{\dim(\sigma)}
\ell^*(\langle \sigma\rangle; t^{-1})
.
\end{align*}
Thus we have
\[
t^{\dim(\Delta)}h^*(\Delta; t^{-1}) 
=
t^{\dim(\Delta)}
\sum_{\sigma \in \Delta}
\ell^*(\langle \sigma\rangle; t^{-1})
\cdot 
h(\lk_\Delta \sigma; t^{-1}).
\]
The result then follows from 
canceling 
$t^{\dim(\Delta)}$
and
replacing $t^{-1}$ by $t$.
\end{proof}

Using the decomposition theorem for the $h^*$-polynomial,
we define the following 
multivariable polynomials.

\begin{defn}[{\cite[Definition~7.5]{KS16a}}]
Let $\Delta$ be a quasi-convex Gorenstein fan.
The \emph{mixed $h^*$-polynomial} of $\Delta$ is defined to be
\[
h^*(\Delta; u, v) = \sum_{\sigma \in \Delta} v^{\dim(\sigma)}\ell^*(\langle \sigma \rangle; u v^{-1}) \cdot
h(\lk_\Delta \sigma; uv).
\]
\end{defn}

Two Gorenstein fans $\Sigma$ and $\Delta$ are said to have \textit{equal degree map} if 
$|\Sigma| = |\Delta|$ and the degree map of $\Sigma$ is equal to that of $\Delta$.

\begin{defn}[{\cite[Definition~8.2]{KS16a}}]
Let $\pi\colon \Sigma\to \Delta$ be a fan subdivision between quasi-convex Gorenstein fans with equal degree map.
The \emph{limit mixed $h^*$-polynomial} of the pair $(\Delta, \Sigma)$ is defined to be
\[
h^*(\Delta, \Sigma; u, v) = \sum_{\sigma\in \Sigma}
v^{\dim(\sigma)} \ell^*(\langle \sigma \rangle; u v^{-1})
\cdot
 h(\lk_\Sigma \sigma; uv).
\]
\end{defn}

Given a fan subdivision $\pi\colon \Sigma\to \langle \sigma \rangle$ between Gorenstein fans with equal degree map,
the \emph{local limit mixed $h^*$-polynomial} of $(\langle \sigma \rangle, \Sigma)$ is defined to be
\cite[Definition~8.2]{KS16a}
\[
\ell^*(\langle \sigma\rangle, \Sigma; u, v) = \sum_{\tau\in \Sigma}
 \ell^*(\langle \tau\rangle; u v^{-1})
 \cdot
 \ell^h_{\lk_{\langle \sigma\rangle} \pi(\tau)} (\lk_\Sigma \tau; uv),
\]
where $\ell^h_{\lk_{\langle\sigma\rangle} \pi(\tau)}(\lk_\Sigma \tau; uv)$ is the local $h$-polynomial of the proper fan morphism
$\pi_\tau \colon \lk_\Sigma \tau \to \lk_{\langle \sigma \rangle} \pi(\tau)$; 
see Section~\ref{sec:minextsheaf}.

\begin{defn}[{\cite[Definition~9.1]{KS16a}}]\label{defn:refinedlimit}
Let $\pi\colon \Sigma\to \Delta$ be a fan subdivision between quasi-convex Gorenstein fans with equal degree map.
The \emph{refined limit mixed $h^*$-polynomial} of the pair $(\Delta, \Sigma)$ is defined to be
\[
h^*(\Delta, \Sigma; u, v, w) = \sum_{\sigma\in \Delta}
w^{\dim(\sigma)}
\ell^*(\langle \sigma\rangle, \pi^{-1}(\langle \sigma\rangle); u, v)
\cdot
h(\lk_\Delta \sigma; u v w^2).
\]
\end{defn}

We have the following commutative diagram showing specializations \cite[Theorem~9.2]{KS16a}:
\[
\begin{tikzcd}
h^*(\Delta, \Sigma; u, v, w) \arrow[r, mapsto]{}{\substack{u\mapsto u w^{-1}\\ v\mapsto 1}}
\arrow[d, mapsto]{}{\substack{w\mapsto 1}}
& h^*(\Delta; u, w) \arrow[d, mapsto]{}{\substack{w\mapsto 1}}\\
h^*(\Delta, \Sigma; u, v)\arrow[r, mapsto]{}{\substack{v\mapsto 1}}  & h^*(\Delta; u).
\end{tikzcd}
\]

\subsection{The Ehrhart sheaf}

In this section, 
we define the Ehrhart sheaf, 
first for a simplicial Gorenstein fan 
and 
then for a general Gorenstein fan.
We note that our definition of the Ehrhart sheaf is slightly different than the one in \cite[Section~3]{Kar08}.

\begin{defn}[{\cite[Section~3]{Kar08}}]\label{defn:GradedEhrhartSheaf}
Let $\Delta$ be a simplicial Gorenstein fan with degree map $G$ in $\R^d$.
We define the \emph{Ehrhart sheaf} $\cE_\Delta$, a sheaf of $\Z_{\geq 0}$-graded $\cA$-modules, as follows.
\begin{itemize}
\item For a cone $\sigma$, 
let $\cE_\Delta(\sigma)$ be the vector space with basis elements
$\{x^\lambda : \lambda\in \sigma \cap \Z^d\}$, 
equipped with the following $A_\sigma$-action $A_\sigma \times \cE_\Delta(\sigma) \to \cE_\Delta(\sigma)$:
for any linear function $L$ on $\Span(\sigma)$, we set
\[
L\cdot x^\lambda = \sum_{i=1}^{\dim(\sigma)} L(v_i) x^{v_i + \lambda},
\]
where $v_1, \dots, v_{\dim(\sigma)}$ are the primitive lattice generators of $\sigma$.
Since $A_\sigma$ is generated by linear functions on $\Span(\sigma)$,
this defines an $A_\sigma$-action.

\item 
For the grading, 
the element $x^\lambda \in \cE_\Delta(\sigma)$ 
is defined to be homogeneous of degree $G(\lambda) \in \Z_{\geq 0}$.

\item For cones $\sigma\geq \tau$, the restriction map $\res_{\sigma, \tau}\colon \cE_\Delta(\sigma) \to \cE_\Delta(\tau)$ is the graded linear map defined by
\[
\res_{\sigma, \tau}(x^\lambda) =
\begin{cases}
x^\lambda & \text{if } \lambda\in \tau \cap  \Z^d, \\
0 & \text{otherwise.}
\end{cases}
\]
\end{itemize}
\end{defn}

This indeed defines a sheaf of $\Z_{\geq 0}$-graded $\cA$-modules,
since 
the restriction maps of $\cE_\Delta$ are compatible with that of $\cA$:
for cones $\sigma\geq \tau$ and a linear function $L$ on $\Span(\sigma)$,
we have
\[
\res_{\sigma, \tau}(L\cdot x^\lambda)
=
\res_{\sigma, \tau}\Big(\sum_{i=1}^{\dim(\sigma)} L(v_i) x^{v_i + \lambda} \Big)
=
\sum_{j=1}^{\dim(\tau)} L(v_{i_j}) x^{v_{i_j} + \lambda} = L|_{\Span(\tau)}\cdot \res_{\sigma, \tau}(x^\lambda),
\]
where $v_{i_1}, \dots, v_{i_{\dim(\tau)}}$ are the primitive lattice generators of $\tau$.

\begin{lem}[{\cite[Section~3]{Kar08}}]
Let $\Delta$ be a simplicial Gorenstein fan.
Then the Ehrhart sheaf $\cE_\Delta$ is a pure sheaf.
\end{lem}

\begin{proof}
We show the flabbiness of $\cE_\Delta$.
It suffices to prove the surjectivity of the restriction map $\res\colon \cE_\Delta(\sigma) \to \cE_\Delta(\partial \sigma)$ whenever $\sigma$ is a cone.
By construction, 
the underlying vector space of $\cE_\Delta(\partial \sigma)$ has basis elements 
$\{x^\lambda : \lambda\in |\partial \sigma| \cap \Z^d\}$.
Since $\lambda\in |\partial \sigma| \cap \Z^d$ implies $\lambda \in \sigma \cap \Z^d$ and for such $\lambda$ we have $\res(x^\lambda) = x^\lambda$, 
the restriction map is surjective.

Now we show that each $\cE_\Delta(\sigma)$ is a finitely generated free $A_\sigma$-module.
Let $\BBox(\sigma) := [0, 1)v_1 + \dots + [0, 1)v_{\dim(\sigma)} \subseteq \R^d$,
where 
$v_1, \dots, v_{\dim(\sigma)}$
are the 
primitive lattice generators of $\sigma$.
Since $\sigma$ is simplicial, a lattice point $\lambda$ in $\sigma$ can be uniquely written as
\[
\lambda = b_\lambda + \sum_{i=1}^{\dim(\sigma)} a_i v_i
\]
for some $b_\lambda \in \BBox(\sigma)\cap  \Z^d$ and $a_i \in \Z_{\geq 0}$.
Hence $\cE_\Delta(\sigma)$ is isomorphic to the finitely generated free $A_\sigma$-module $\oplus_{\lambda\in \BBox(\sigma) \cap \Z^d} A_\sigma x^{\lambda}$.
\end{proof}

We now define the Ehrhart sheaf for a non-simplicial Gorenstein fan.

\begin{defn}[{\cite[Section~3]{Kar08}}]\label{defn:NonSimplicialEhrhartSheaf}
Let $\Delta$ be a non-simplicial Gorenstein fan.
We pick a simplicial Gorenstein refinement $\tilde{\Delta}$ of $\Delta$ 
that has degree map equal to that of $\Delta$.
Let $\pi\colon \tilde{\Delta}\to \Delta$ be the corresponding fan subdivision.
The Ehrhart sheaf $\cE_\Delta$ of $\Delta$ is then defined to be $\pi_* \cE_{\tilde{\Delta}}$.
\end{defn}

Such a Gorenstein fan $\tilde{\Delta}$ always exists, 
since every fan admits a simplicial refinement without adding new rays \cite[Section~2.2.1]{DLRS10}.
Since no new rays are added, 
the degree map $G_\Delta$ of $\Delta$ 
sends the primitive lattice generator of each ray in $\tilde{\Delta}$ to $1$.
As a result, each cone $\sigma$ in $\tilde{\Delta}$ is Gorenstein with $G_\sigma = G_\Delta|_{\sigma}$.
Hence $\tilde{\Delta}$ is a Gorenstein fan and 
its degree map is equal to that of $\Delta$.

The Ehrhart sheaf $\cE_\Delta$ of a non-simplicial fan $\Delta$ is also a pure sheaf, since it is the direct image sheaf of a pure sheaf under a fan subdivision. 
We note that the Ehrhart sheaf does not depend on the choice of a refinement,
since any two refinements give pure sheaves that
admit the same decomposition into shifts of simple sheaves \cite[p.~145]{Kar08}.

\begin{prop}\label{p:Ehrhartcombin}
The following are true:
\begin{enumerate}
\item \label{item:Ehrhartcombin1}
For a quasi-convex Gorenstein fan $\Delta$ in $\R^d$, we have
$h^*(\Delta; t) = P(\ov{\cE_\Delta(\Delta)}; t)$.

\item For a cone $\sigma$ in a Gorenstein fan $\Delta$, we have
$L(\cE_\Delta, \sigma; t)
= \ell^*(\langle\sigma\rangle; t)$.
\end{enumerate}
\end{prop}

\begin{proof}
The first statement was proved in \cite[Proposition~3.3]{Kar08}:
We first show that the underlying graded vector space of $\cE_\Delta(\Delta)$ 
has
basis elements 
$\{ x^\lambda : \lambda\in |\Delta| \cap \Z^d \}$,
where the element $x^\lambda$ has degree $G(\lambda)$.
If $\Delta$ is simplicial, then this follows from the construction of the Ehrhart sheaf.
If $\Delta$ is non-simplicial, 
then $\cE_\Delta(\Delta)$ is isomorphic to 
$\cE_{\tilde{\Delta}}(\tilde{\Delta})$
for some simplicial Gorenstein refinement $\tilde{\Delta}$ of $\Delta$ 
that has degree map equal to that of $\Delta$.
Since $|\Delta| = |\tilde{\Delta}| $,
the result follows from the simplicial case.

Now, we note that the Hilbert series of the 
graded vector space $\cE_\Delta(\Delta)$ is equal to the Ehrhart series of $\Delta$.
Since $\Delta$ is quasi-convex and $\cE_\Delta$ is a pure sheaf, by Proposition~\ref{p:quasiconvex} the $A$-module $\cE_\Delta(\Delta)$ is free.
Hence \eqref{eq:freeimplieshilbpoin} implies that the Poincar\'e polynomial of $\ov{\cE_\Delta(\Delta)}$ satisfies
\[
\frac{P(\ov{\cE_\Delta(\Delta)}; t)}{(1-t)^d} =
\Hilb(\cE_\Delta(\Delta); t) =  \Ehr_\Delta(t)
= \frac{h^*(\Delta; t)}{(1-t)^d},
\]
and the result follows.

The second statement when $\Delta$ is simplicial was observed in \cite[p.~140]{Kar08}.
For the general case, 
we assume without loss of generality that $\Delta = \langle\sigma \rangle$.
We then apply \eqref{eq:Poincar\'erecursion} to the Ehrhart sheaf and get
\begin{equation*}
P(\ov{\cE_\Delta(\Delta)}; t) = 
\sum_{\tau\in \Delta} 
L(\cE_\Delta, \tau; t) \cdot
P(\ov{\cL_\tau(\Delta)}; t).
\end{equation*}
On the other hand,
we note that the fan $\Delta = \langle \sigma \rangle$ is quasi-convex in $\Span(\sigma)$.
Thus \eqref{eq:hstardecomp} 
implies that 
\begin{align*}
h^*(\Delta; t) 
& =
\sum_{\tau \in \Delta} 
\ell^*(\langle \tau \rangle; t)
\cdot
h(\lk_\Delta \tau; t).
\end{align*}
By an induction on $\dim(\sigma)$,
we may assume 
$L(\cE_\Delta, \tau; t) = \ell^*(\langle \tau \rangle; t)$
when $\dim(\tau) < \dim(\sigma)$.
Now, 
since 
$P(\ov{\cL_\sigma(\langle \sigma \rangle)}; t) 
= 1
= 
h(\lk_{\langle \sigma \rangle} \sigma; t)$, 
we have
\begin{align*}
L(\cE_\Delta, \sigma; t)  
&=
P(\ov{\cE_\Delta(\Delta)}; t) 
-
\sum_{o \leq \tau < \sigma} 
L(\cE_\Delta, \tau; t) \cdot
P(\ov{\cL_\tau(\Delta)}; t)\\
&=
h^*(\Delta; t)
-
\sum_{o \leq \tau < \sigma} 
\ell^*(\langle \tau \rangle; t)
\cdot
h(\lk_\Delta \tau; t)\\
& = 
\ell^*(\langle \sigma \rangle; t),
\end{align*}
where the second equality uses 
Statement~\eqref{item:Ehrhartcombin1}
and
Proposition~\ref{p:hpolyPoincar\'e}.
\end{proof}

Recall from Section~\ref{sec:polysfromfiltrations} the definitions of the Hodge--Deligne polynomials.

\resmixedhstar

\begin{proof}
We pick a decomposition of $\cE_\Sigma$ into shifts of simple sheaves on $\Sigma$:
\[
\cE_\Sigma  = \bigoplus_{\alpha} \cL_{\sigma_\alpha}[-j_\alpha].
\]
By Lemma~\ref{lem:mixedpolyslimit}, the refined limit Hodge--Deligne polynomial is given by
\begin{align*}
& E(\ov{\pi_*\cE_\Sigma(\Delta)}, F_\bullet, W^\bullet, M^\bullet; u, v, w)\\
& \qquad= 
\sum_{\alpha} u^{j_\alpha} v^{\dim(\sigma_\alpha) - j_\alpha}
\sum_{\tau \in \Delta} w^{\dim(\tau)}
L(\pi_*\cL_{\sigma_\alpha}, \tau, uv)
\cdot P(\ov{\cL_\tau(\Delta)}; uvw^2)\\
&\qquad = 
\sum_{\sigma\in \Sigma}
v^{\dim(\sigma)}
\bigg(
\sum_{\substack{\alpha \\ \sigma_\alpha = \sigma}}
u^{j_\alpha} v^{- j_\alpha}
\bigg)
\sum_{\tau \in \Delta} w^{\dim(\tau)}
L(\pi_*\cL_{\sigma}, \tau, uv)
\cdot
 P(\ov{\cL_\tau(\Delta)}; uvw^2)\\
& \qquad = 
\sum_{\tau \in \Delta} w^{\dim(\tau)}
\bigg(\sum_{\sigma\in \Sigma}
v^{\dim(\sigma)}
L(\cE_\Sigma, \sigma; u v^{-1})
\cdot
L(\pi_*\cL_{\sigma}, \tau, uv) \bigg)
\cdot
P(\ov{\cL_\tau(\Delta)}; uvw^2),
\end{align*}
where the last equality uses the definition of the local Poincar\'e polynomial.
By Proposition~\ref{p:Ehrhartcombin} and  Proposition~\ref{p:hpolyPoincar\'e}, we have
\begin{align*}
L(\cE_\Sigma, \sigma; u v^{-1}) & = \ell^*(\langle\sigma \rangle); u v^{-1}), \\
L(\pi_*\cL_{\sigma}, \tau, uv )& =  \ell^h_{ \lk_{\langle\tau \rangle} \pi(\sigma)}(\lk_{\pi^{-1}(\langle \tau \rangle)} \sigma; uv),\\
P(\ov{\cL_\tau(\Delta)}; uvw^2)& =  h(\lk_\Delta \tau; uvw^2),
\end{align*}
hence the result follows.
\end{proof}

We obtain the following two corollaries by considering specializations.

\begin{cor}
Let $\Delta$ be a quasi-convex Gorenstein fan in $\R^d$.
Then the Hodge--Deligne polynomial of $(\ov{\cE_\Delta(\Delta)}, F_\bullet, W^\bullet)$ is equal to the mixed $h^*$-polynomial of $\Delta$.
\end{cor}

\begin{cor}
Let $\pi\colon \Sigma\to \Delta$ be a fan subdivision between quasi-convex Gorenstein fans in $\R^d$ with equal degree map.
The limit Hodge--Deligne polynomial of $(\ov{\pi_*\cE_\Sigma(\Delta)}, F_\bullet, M^\bullet)$ is equal to the limit mixed $h^*$-polynomial of $(\Delta, \Sigma)$.
\end{cor}

\begin{remark}
As observed in \cite[Theorem~9.4]{KS16a},
given a projective fan subdivision $\pi\colon \Sigma\to \De$,
if we write $h^*(\Delta, \Sigma; u, v, w)  = 1+ u v w^2 \sum h^*_{p, q, r} u^p v^q w^r$,
then for $0\leq k \leq r \leq \dim(\Delta) - 2$,
the sequence $\{h_{k+i, i, r} : 0\leq i \leq r-k\}$ is symmetric, nonnegative and unimodal.
Alternatively, the nonnegativity follows immediately from Theorem~\ref{t:refinedEhrhart}.
The unimodality and the symmetry follow from the relative hard Lefschetz theorem for a {projective} proper fan morphism \cite[Theorem~1.1]{Kar19}.
\end{remark}

\section{The mixed \texorpdfstring{$cd$}{cd}-index from a different minimal extension sheaf}\label{sec:cdindex}

In this section, 
we first introduce the $cd$-index 
following 
\cite[Chapter~3]{Sta12};
see also \cite{Bay21} for a survey.
We then introduce 
the mixed $cd$-index,
an enrichment of the $cd$-index,
following \cite[Section~6]{DKT20}.
After that,
we study the $(\Z_{\geq 0})^d$-graded pure sheaves with structure sheaf $\cC$ following \cite[Section~2]{Kar06},
and introduce the \textit{$t$-degree},
a specialization of the $(\Z_{\geq 0})^d$-grading,
to justify the setup of the weight sheaves in this context.
At the end, we define the weight filtration and
show that
the Hodge--Deligne polynomial is equal to the image of 
the mixed $cd$-index under the injective linear map $\eta'$.

Before we start, we note that in this section 
the sheaf $\cC$ which is defined below plays the role of the structure sheaf.

\subsection{The \texorpdfstring{$cd$}{cd}-index}\label{subsec:cdindex}

Let $\Delta$ be a purely $n$-dimensional fan in $\R^d$.
For a subset $S\subseteq \{1, \dots, n\}$,
the \textit{flag $f$-number} $f_S$ of $\Delta$ is defined to be the number of flags
$\{o < \sigma_1 < \dots < \sigma_n\}$ of $\Delta$
satisfying $\{\dim (\sigma_1), \dots, \dim (\sigma_n)\} = S$.

Let $\R\langle a, b \rangle$ be the graded polynomial ring in non-commuting variables $a$ and $b$,
where $\deg(a) = \deg(b) = 1$.
The \emph{$ab$-index} $\Psi_\Delta(a, b)$
of $\Delta$ is defined to be 
$\Psi_\Delta(a, b) := \sum_{S\subseteq \{1, \dots, n\}} f_S u_S \in \R\langle a, b \rangle$,
where $u_S = u_1 \cdots u_n$ is defined by 
\[
u_i = \begin{cases}
a-b & \text{if } i\notin S, \\
b & \text{if } i\in S.
\end{cases}
\]
In particular,
the $ab$-index is a homogeneous polynomial of degree $\dim(\Delta)$.
%By convention, we define the $ab$-index of the empty fan to be $0$.

Let $\R \langle c, d\rangle$ be the graded polynomial ring in non-commuting variables $c$ and $d$,
where $\deg(c) = 1$ and $\deg(d) = 2$.
Fine conjectured and Bayer and Klapper \cite[Theorem~4]{BK91} proved that as a consequence of the generalized Dehn--Sommerville equations for the flag $f$-numbers \cite[Theorem~4.1]{BB85},
the $ab$-index of a complete fan $\Delta$
(or more generally, the boundary of an \textit{Eulerian poset})
satisfies 
\begin{equation*}\label{eq:cdIndexIscdExpressible}
\Psi_\Delta(a, b) = \Phi_\Delta(a+b, ab+ba)
\end{equation*}
for a unique polynomial $\Phi_\Delta(c, d)\in \R \langle c, d\rangle$ called the \emph{$cd$-index}.
In particular,
the $cd$-index is a homogeneous polynomial of degree $\dim(\Delta)$.

\begin{remark}
Our definitions of the $ab$-index and 
the $cd$-index do not require the fan to have a unique maximal element, unlike the one in the literature. 
\end{remark}

\begin{defn}[{\cite[p.~231]{EK07}}]
\label{defn:LocalcdIndex}
The \emph{local $cd$-index} $\ell^\Phi_\Delta(c, d)$ of a non-complete quasi-convex fan $\Delta$ is defined to be the unique polynomial in $\R\langle c, d\rangle$ satisfying 
\[
\Psi_\Delta(a, b)
=
\ell^\Phi_\Delta(a+b, ab+ba)
 +
\Psi_{\partial \Delta}(a, b) \cdot a.
\]
\end{defn}

%Equivalently, 
%we can define the local $cd$-index
%to be \cite[Lemma~3.5]{DKT20}
%\[
%\ell^\Phi_\Delta(c, d) 
%=
%\Phi_{\partial \widetilde{\Delta}}(c, d)
%-
%\Phi_{\partial \Delta}(c, d) \cdot c,
%\]
%where 
%$\partial \widetilde{\Delta}$ is the boundary of a \textit{Gorenstein$^*$} poset, 
%obtained by adding an element to $\Delta$ with boundary $\partial \Delta$ \cite[Section~3.1]{Kar06}.
%Here,
%the $cd$-index $\Phi_{\partial \Delta}(c, d)$ is well-defined,
%since $\partial \Delta$ is 
%the boundary of an Eulerian poset \cite[Section~3]{Kar06}. 

We then define 
the \emph{$cd$-index} $\Phi_\Delta(c, d) \in \R\langle c, d\rangle$ of a non-complete quasi-convex fan $\Delta$ to be {\cite[p. 231]{EK07}} 
\[
\Phi_\Delta(c, d) =  \ell^\Phi_\Delta(c, d)  + \Phi_{\partial \Delta}(c, d).
\]
Here, 
the $cd$-index $\Phi_{\partial \Delta}$ is well-defined,
since the underlying poset of $\partial \Delta$ is the boundary of an Eulerian poset \cite[Section~3.1]{Kar06}. 

There is a decomposition theorem for the $cd$-index \cite[Theorem~2.7]{EK07}:
for a fan subdivision $\pi\colon \Sigma\to \Delta$ between quasi-convex fans, we have
\begin{equation}\label{eq:cddecomp}
\Phi_\Sigma(c, d) = \sum_{\sigma\in \Delta} \ell^\Phi_{\pi^{-1}(\langle \sigma \rangle)}(c, d)
\cdot
 \Phi_{\link_\Delta \sigma}(c, d).
\end{equation}
Here, 
the preimage fan $\pi^{-1}(\langle \sigma \rangle)$ is supported on $\sigma$, hence it is a non-complete quasi-convex fan in $\Span(\sigma)$ and its local $cd$-index is well-defined. 
This theorem was proved using the theory of pure sheaves in \cite[Theorem~2.7]{EK07}.
%See \cite[Corollary~5.3]{DKT20} for a combinatorial proof.

We consider the $\Z_{\geq 0}$-graded vector space $\R\langle c', d'\rangle \otimes_\R \R \langle c, d\rangle$, where $c'$ and $d'$ are non-commuting variables of degree $1$ and $2$, respectively.

\begin{defn}[{\cite[Definition~6.1]{DKT20}}]\label{defn:mixedcd}
For a fan subdivision $\pi\colon \Sigma\to \Delta$ between quasi-convex fans, the \emph{mixed $cd$-index} $\Omega_\pi(c', d', c, d) \in \R\langle c', d'\rangle \otimes_\R \R \langle c, d\rangle$ is defined to be
\[
\Omega_\pi(c', d', c, d) = \sum_{\sigma\in \Delta} \ell^\Phi_{\pi^{-1}(\langle \sigma \rangle)}(c', d')  \otimes
\Phi_{\link_\Delta \sigma}(c, d).
\]
\end{defn}

\begin{remark}
Bayer and Ehrenborg used the fact that the $cd$-index is a \textit{coalgebra homomorphism} \cite[Proposition~3.1]{ER98} to show that the $cd$-index determines the toric $h$-polynomial \cite[Proposition~7.1]{BE00}.
Later, Dornian, Katz and the author used the fact that the mixed $cd$-index is a \textit{comodule homomorphism} \cite[Remark~6.9]{DKT20} to show that the mixed $cd$-index determines the mixed $h$-polynomial \cite[Theorem~7.13]{DKT20}.
\end{remark}

\begin{remark}
The original definition \cite[Definition~6.1]{DKT20} uses a variable $e$ of degree $-1$ to homogenize the polynomial.
Our definition is obtained by setting $e\mapsto 1$.
\end{remark}

\subsection{Pure sheaves for the \texorpdfstring{$cd$}{cd}-index}

The results in this subsection can be found in \cite{Kar06}.
We also refer the reader to \cite[Chapter~8]{MS05} for background in multigraded modules.

Let $C_d := \R[x_1, \dots, x_d]$ be the $(\Z_{\geq 0})^d$-graded polynomial ring, 
where $\deg(x_i)  = e_i \in (\Z_{\geq 0})^d$.
Let $\mmax$ be the homogeneous maximal ideal generated by $x_1, \dots, x_d$.
For a finitely generated $(\Z_{\geq 0})^d$-graded $C_d$-module $M$, we write $\ov{M}$ for the $(\Z_{\geq 0})^d$-graded vector space $M\otimes_{C_d} C_d/\mmax$, 
and let $\rho\colon M\to \ov{M}$ be the corresponding quotient map.

The \emph{structure sheaf} $\cC$ on a fan $\Delta$ in $\R^d$ is a sheaf of $(\Z_{\geq 0})^d$-graded rings defined as follows.
For a cone $\sigma$, we set $\cC(\sigma) :=  \R[x_1, \dots, x_{\dim(\sigma)}]$, where $\deg(x_i) = e_i \in (\Z_{\geq 0})^d$.
For cones $\sigma \geq \tau$, the restriction map $\res_{\sigma, \tau}\colon \cC(\sigma) \to \cC(\tau)$ is
the graded ring map defined by
\[
\res_{\sigma, \tau}(x_i) = \begin{cases}
x_i & \text{if } i\leq \dim(\tau),\\
0 & \text{otherwise.}
\end{cases}
\]
This defines a sheaf of $(\Z_{\geq 0})^d$-graded rings.

A sheaf $\cF$ is a \textit{sheaf of $(\Z_{\geq 0})^d$-graded $\cC$-modules} if $\cF(U)$ is a $(\Z_{\geq 0})^d$-graded $\cC(U)$-module
whenever $U$ is an open set and 
the restriction maps of $\cF$ are graded module maps compatible with that of $\cC$.
In particular,
each $\cF(U)$ is a graded $C_d$-module and 
each restriction map of $\cF$ is a graded $C_d$-module map,
since there is a canonical ring map $C_d\to \cC(\sigma)$ whenever $\sigma$ is a cone.

For notational convenience, 
we write 
$\cF(\Delta, \partial \Delta)$
for 
$\ker(\res \colon \cF(\Delta) \to  \cF(\partial \Delta))$.
For a cone $\sigma\in \Delta$,
we write 
$\cF(\sigma, \partial \sigma)$
for 
$\ker(\res \colon \cF(\sigma) \to  \cF(\partial \sigma))$.

\begin{defn}
A sheaf $\cF$ of $(\Z_{\geq 0})^d$-graded $\cC$-modules on a fan $\Delta$ is called a \emph{pure sheaf} if $\cF$ is flabby and 
$\cF(\sigma)$ is a finitely generated free $\cC(\sigma)$-module whenever $\sigma$ is a cone.
\end{defn}

If $\cF$ is a pure sheaf on $\Delta$,
then both $\cF(\Delta)$ and $\cF(\Delta, \partial \Delta)$
are finitely generated as $C_d$-modules \cite[p.~8]{BBFK02}.

\begin{defn}\label{defn:MultiGradedSimpleSheaf}
For a cone $\sigma$ in a fan $\Delta$, a \emph{simple sheaf} $\cL_\sigma$ based at $\sigma$ is a sheaf of $(\Z_{\geq 0})^d$-graded $\cC$-modules satisfying the following conditions:
\begin{itemize}
\item \textit{Normalization:} For each $\tau\in \Delta$ of dimension at most $\dim(\sigma)$,
\[
\cL_\sigma(\tau) = \begin{cases}
\cC(\tau) & \text{if } \tau  = \sigma,\\
0 & \text{otherwise.}
\end{cases}
\]

\item \textit{Local freeness:} For each cone $\tau \in \Delta$, the $\cC(\tau)$-module $\cL_\sigma(\tau)$ is free.

\item \textit{Local minimal extension:} For each cone $\tau \in \Delta$ of dimension strictly greater than $\dim(\sigma)$, the restriction map $\res\colon \cL_\sigma(\tau) \to\cL_\sigma(\partial \tau)$ induces an isomorphism of graded vector spaces:
\[
\begin{tikzcd}
\ov{\res}\colon \ov{\cL_\sigma(\tau)}
\arrow{r}{\sim}
& \ov{\cL_\sigma(\partial \tau)}.
\end{tikzcd}
\]
\end{itemize}
\end{defn}

As noted in \cite[Section~2.1]{Kar06}, 
for each cone $\sigma$ the simple sheaf $\cL_\sigma$ exists 
and is unique up to an isomorphism.
We also call the simple sheaf $\cL_o$ the \emph{minimal extension sheaf} $\cL_\Delta$ of $\Delta$.

\begin{remark}
\label{remark:cdPureSheavesDependPosetOnly}
As noted in \cite[p.~704]{Kar06}, 
the minimal extension sheaf $\cL_\Delta$ 
with structure sheaf $\cC$
depends on the underlying poset of the fan instead of its realization:
Suppose $\Delta$ in $\R^{d}$ and $\Sigma$ in $\R^{d'}$ are \textit{combinatorially equivalent} fans, i.e., fans with isomorphic underlying posets.
Without loss of generality, we assume that $d \leq d'$.
Then for each $(j_1, \dots, j_{d'}) \in (\Z_{\geq 0})^{d'}$,
we have
%the $(j_1, \dots, j_{d'})$-multigraded component of 
%$\ov{\cL_{\Sigma}(\Sigma)}$ is given by 
\[
\ov{\cL_{\Sigma}(\Sigma) }
^{(j_1, \dots, j_{d'})}
=
\begin{cases}
\ov{\cL_{\Delta}(\Delta) }
^{(j_1, \dots, j_{d})}
& \text{if } j_{d + 1} = \dots = j_{d'} = 0,\\
0 & \text{otherwise.}
\end{cases}
\]
\end{remark}

There is a decomposition theorem for the pure sheaves with structure sheaf $\cC$ \cite[p.~227]{EK07}.
It states that 
a pure sheaf $\cF$ on a fan $\Delta$ in $\R^d$
admits a direct sum decomposition into shifts of simple sheaves:
\[
\cF \simeq \bigoplus_{\sigma\in \Delta}
K_\sigma \otimes_\R \cL_\sigma,
\]
where $K_\sigma$ is the $(\Z_{\geq 0})^{d}$-graded vector space 
$\ker(\ov{\res}\colon \ov{\cF(\sigma)} \to \ov{\cF(\partial \sigma)})$.
Alternatively we can write $\cF \simeq \bigoplus_\alpha \cL_{\sigma_\alpha}[-{j_\alpha}]$,
where each $\sigma_\alpha$ is a possibly repeated cone in $\Delta$ and 
each ${j_\alpha} \in (\Z_{\geq 0})^d$ is a possibly repeated lattice point.
Here, given $j\in (\Z_{\geq 0})^d$, the \textit{shifted sheaf} $\cF[-j]$ is defined by 
\[
\cF[-j](U)^\lambda = \cF(U)^{\lambda  - j}
\]	 
whenever $U$ is an open set 
and $\lambda \in (\Z_{\geq 0})^d$.

Karu \cite[Section~2.2]{Kar06} verified that 
the duality results \cite[Section~4 and 6]{BBFK02} are flexible enough to hold for the pure sheaves with structure sheaf $\cC$:
For a quasi-convex fan $\Delta$ in $\R^d$,
both $\cL_\Delta(\Delta)$ and $\cL_\Delta(\Delta, \partial \Delta)$
%$ = \ker(\res\colon \cL_\Delta(\Delta)\to \cL_\Delta(\partial \Delta))$ 
are 
$(\Z_{\geq 0})^d$-graded free $C_d$-modules
dual to each other in the following sense:
\begin{equation}\label{eq:cdduality}
\cL_\Delta(\Delta, \partial \Delta)
 = 
 \Hom_{C_d} \big(
 \cL_\Delta(\Delta), C_d[-e_1 - \dots - e_d]
\big).
\end{equation}

\subsection{The \texorpdfstring{$t$}{t}-degree}

We introduce a $\Z_{\geq 0}$-grading that is a specialization of the $(\Z_{\geq 0})^d$-grading.
Let $T\colon (\Z_{\geq 0})^d \to \Z_{\geq 0}$ be an additive map that sends each $e_i \in (\Z_{\geq 0})^d$ to $2^{i-1} \in \Z_{\geq 0}$.
For a homogeneous element $f$ of degree $\deg(f) \in (\Z_{\geq 0})^d$, its \emph{$t$-degree} $\tdeg(f)$ is defined to be $T(\deg(f))$.
For a $(\Z_{\geq 0})^d$-graded vector space $V$, we define its \emph{$t$-Poincar\'e polynomial} to be
\[P^t(V; t) = \sum_{\lambda\in (\Z_{\geq 0})^d} \dim_\R (V^{\lambda}) t^{T(\lambda)}.\]

Given an $n$-dimensional cone $\sigma$, we define its \emph{$t$-dimension} to be $\tdim(\sigma) =  2^{n}-1$.
Then the $t$-dimension $\tdim(\Delta)$ of a fan $\Delta$ is defined to be the maximum of the $t$-dimensions of its cones.

The following lemma relates the simple sheaf $\cL_\sigma$ on $\Delta$ and the minimal extension sheaf $\cL_{\lk_\Delta\sigma}$ of $\lk_\Delta\sigma$.
%While we believe the lemma is known to the community, 
%for the purpose of completeness we provide a proof here. 

\begin{lem}\label{lem:SimpleSheafOfLinkIsShiftInDegree}
Let $\Delta$ be a fan in $\R^d$ and $\sigma \in \Delta$ be a cone of dimension $k$.
Then 
\begin{align*}
P^t(\ov{\cL_\sigma(\Delta)}; t) 
& =
P^t(\ov{\cL_{\lk_\Delta \sigma}(\lk_\Delta \sigma)}; t^{2^{k}}), \\
P^t(\ov{\cL_\sigma(\Delta, \partial \Delta)}; t) 
& =
P^t(\ov{\cL_{\lk_\Delta \sigma}(\lk_\Delta \sigma, \partial \lk_\Delta \sigma)}; t^{2^k}).
\end{align*}
\end{lem}

\begin{proof}
Let $i\colon \lk_\Delta\sigma\to \Delta$ be the injective map sending $\ov{\tau}$ to $\tau$.
This is a continuous map.
Thus the direct image sheaf $i_* \cL_{\lk_\Delta\sigma}$ is a sheaf of $(\Z_{\geq 0})^{d-k}$-graded $i_*\cC_{\lk_\Delta \sigma}$-modules on $\Delta$.

Recall that $C_k$ is the $(\Z_{\geq 0})^k$-graded ring $\R[x_1, \dots, x_k]$.
Let 
$C_k \otimes_\R i_* \cL_{\lk_\Delta\sigma}$
be 
the $(\Z_{\geq 0})^{d}$-graded sheaf,
% of  $C_k \otimes_\R i_* \cC_{\lk_\Delta\sigma}$-modules,
where for each open set $U$
and each $(j_1, \dots, j_d) \in (\Z_{\geq 0})^d$, 
\begin{equation*}\label{eq:StackingGrading}
(
C_k \otimes_\R i_* \cL_{\lk_\Delta\sigma}
)(U) 
^{(j_1, \dots, j_d)}
: =
C_k^{(j_1, \dots, j_k)}
\otimes_\R
i_* \cL_{\lk_\Delta\sigma}(U) 
^{(j_{k+1}, \dots, j_d)}.
\end{equation*}
While this is a sheaf of 
$C_k \otimes_\R i_* \cC_{\lk_\Delta\sigma}$-modules,
we can consider it as
a sheaf of 
$\cC$-modules:
For a cone $\tau$, 
the stalk
$(C_k \otimes_\R i_* \cC_{\lk_\Delta\sigma})(\tau)$
is nonzero
precisely when $\tau \geq \sigma$.
Now, 
for such a cone $\tau$ 
we have
\[
(C_k \otimes_\R i_* \cC_{\lk_\Delta\sigma})(\tau)
=
C_k \otimes_\R \cC_{\lk_\Delta\sigma}(\ov{\tau})
\simeq 
\cC_{\Delta}(\tau).
\]
%as $(\Z_{\geq 0})^d$-graded rings. 
%where the last isomorphism of rings is given by 
%$x_i \otimes 1\mapsto x_i$
%and 
%$1\otimes x_j \mapsto x_{j+k}$.

We prove that $C_k \otimes_\R i_* \cL_{\lk_\Delta\sigma}$
is the simple sheaf based at $\sigma$
by 
checking the conditions in Definition~\ref{defn:MultiGradedSimpleSheaf}:
\begin{itemize}
\item At $\sigma$, 
we have 
$(C_k\otimes_\R i_* \cL_{\lk_\Delta\sigma})(\sigma) = C_{k}\otimes_\R \R \simeq \cC_{\Delta}(\sigma)$.

\item For each $\tau \geq \sigma$, 
since $\cL_{\lk_\Delta\sigma}(\ov{\tau})$ is a 
free 
$\cC_{\lk_\Delta\sigma}(\ov{\tau})$-module,
the stalk
$(C_{k}\otimes_\R i_* \cL_{\lk_\Delta\sigma})(\tau)$ 
is a free $C_k \otimes_\R \cC_{\lk_\Delta \sigma}(\ov{\tau}) \simeq \cC_\Delta(\tau)$-module.

\item 
The restriction map of $C_k\otimes_\R i_* \cL_{\lk_\Delta\sigma}$ 
is the tensor product of the identity map on $C_k$ and the restriction map 
of $i_*\cL_{\lk_\Delta\sigma}$.
Given a cone $\tau > \sigma$,
since 
$C_k \otimes_\R \cC_{\lk_\Delta \sigma}(\ov{\tau}) \simeq \cC_\Delta(\tau)$,
the map induced by the restriction map of 
$C_k\otimes_\R i_* \cL_{\lk_\Delta\sigma}$
is given by 
\[
\Big(
\ov{\id}\colon \ov{C_k} \to \ov{C_k} 
\Big)
\otimes_\R 
\Big(
\ov{\res}\colon 
\ov{\cL_{\lk_\Delta \sigma}(\ov{\tau})}
\to 
\ov{\cL_{\lk_\Delta \sigma}(\partial \ov{\tau})}
\Big).
\]
Since both $\ov{\id}$ and $\ov{\res}$ are isomorphisms of graded vector spaces,
their tensor product is also an isomorphism of graded vector spaces. 
\end{itemize}

%Since $\cL_{\lk_\Delta\sigma}$ is the simple sheaf based at $o \in \lk_\Delta\sigma$,
%we have that 
%$C_k \otimes_\R i_* \cL_{\lk_\Delta\sigma}$ is the simple sheaf based at $\sigma$.
As a result, we have
$\cL_\sigma(\Delta)
\simeq 
(C_k \otimes_\R i_* \cL_{\lk_\Delta\sigma})(\Delta)
$
and 
\begin{align*}
\cL_\sigma(\Delta, \partial \Delta)
& \simeq 
\ker\big( \res\colon (C_k \otimes_\R i_* \cL_{\lk_\Delta\sigma})(\Delta) \to (C_k \otimes_\R i_* \cL_{\lk_\Delta\sigma})(\partial \Delta)\big)
\\
& \simeq 
C_k
\otimes_\R
\cL_{\lk_\Delta\sigma}(\lk_\Delta\sigma, \partial \lk_\Delta\sigma).
\end{align*}
%since $i^{-1}(\partial \Delta) = \partial \lk_\Delta \sigma$.
Here, 
we use 
the algebra fact that given two surjective module maps $f\colon V\to W$ and $g\colon V'\to W'$, we have 
$\ker(f\otimes g) = \ker(f)\otimes V + V' \otimes \ker(g)$; see for example \cite[Theorem~18]{Gri17}.
Now, 
since $C_k \otimes_\R C_{d-k} \simeq C_d$,
% as $(\Z_{\geq 0})^d$-graded rings
%with the grading defined similarly to \eqref{eq:StackingGrading},
we have
\begin{align*}
\ov{\cL_\sigma(\Delta)}
& \simeq 
\R \otimes_\R \ov{\cL_{\lk_\Delta\sigma}(\lk_\Delta\sigma)}, \\
\ov{\cL_\sigma(\Delta, \partial \Delta)}
& \simeq 
\R \otimes_\R \ov{\cL_{\lk_\Delta\sigma}(\lk_\Delta\sigma, \partial \lk_\Delta\sigma)}.
\end{align*}
Since 
the $\R$'s in the first components are $(\Z_{\geq 0})^k$-graded,
the results follow.
\end{proof}

Given a pure sheaf $\cF$ on $\Delta$, the \emph{local $t$-Poincar\'e polynomial} $L^t(\cF, \sigma; t)$ is defined to be $P^t(K_\sigma; t)$, 
where $K_\sigma$
is the $(\Z_{\geq 0})^d$-graded vector space defined above
satisfying $\cF \simeq \oplus_{\sigma\in \Delta} K_\sigma\otimes_\R \cL_\sigma$.
Thus the decomposition theorem for pure sheaves implies
\begin{equation*}\label{eq:cdPoincar\'edecomp}
P^t(\ov{\cF(\Delta)}; t) =
\sum_{\sigma\in \Delta} L^t(\cF, \sigma; t)
\cdot
 P^t(\ov{\cL_\sigma(\Delta)}; t).
\end{equation*}

We proceed to relate the $cd$-index and the $t$-Poincar\'e polynomial.
We first define a linear map $\eta \colon \R\langle c, d\rangle \to \R[t]$ recursively:
Let $\eta(1; t) = 1$.
Given a homogeneous polynomial $f(c, d)$ of degree $k$,
we set
\begin{align*}\label{eq:etaChangeOfVariable}
\begin{split}
\eta \big(f(c, d) \cdot c; t \big) &= \eta \big(f(c, d); t\big) \cdot  (1+t^{2^k}),\\
\eta \big(f(c, d) \cdot d; t \big) &= \eta\big(f(c, d); t\big) \cdot (t^{2^k}+t^{2^{k+1}}).
\end{split}
\end{align*}
For example, we have  $\eta(ccd; t) = (1+t)(1+t^2)(t^4 + t^8)$.

\begin{prop}\label{p:etaIsInjective}
The linear map $\eta$ is injective.
\end{prop}

\begin{proof}
This linear map $\eta$ is just the injective linear map $\phi$ in \cite[Section~2.3]{Kar04} with each $t_i$ replaced by $t^{T(e_i)} = t^{2^{i-1}}$.
Now, 
by the definition of $\eta$, 
each 
$t^{2^{i-1}}$ appears at most once in each monomial term in $\eta(f(c, d); t)$.
Hence 
the operation of replacing each $t_i$ by $t^{2^{i-1}}$ is reversible, and $\eta$ is injective.
\end{proof}

Since $\eta$ is injective, 
the following proposition implies that 
the $cd$-index can be uniquely recovered from the 
$t$-Poincar\'e polynomial.
We believe the proposition is known to the community,
since its content is used in the proof of \cite[Theorem~2.7]{EK07}.
For the purpose of completeness, we provide a proof here.

\begin{prop}\label{p:FactsAbouttPoin}
Let $\Delta$ be a quasi-convex fan in $\R^d$.
Then the following are true:
\begin{enumerate}
\item \label{item:tP1link}
For a cone $\sigma \in \Delta$, we have 
${ \eta}(\Phi_{\lk_\Delta \sigma}(c, d); t^{2^{\dim(\sigma)}})
 = P^t(\ov{\cL_\sigma(\Delta)};t)$. 
 In particular, 
we have
${ \eta}(\Phi_{\Delta }(c, d); t)
 = P^t(\ov{\cL_\Delta(\Delta)};t)$. 

\item \label{item:tPlocal}
If $\Delta$ is supported on a cone $\sigma$,
then we have 
$L^t(\pi_*\cL_\Delta, \sigma; t) = { \eta}(\ell_\Delta^\Phi(c, d); t)$,
where
$\pi\colon \Delta \to \langle \sigma\rangle$ is the 
corresponding fan subdivision.
\end{enumerate}
\end{prop}

\begin{proof}
For Statement \eqref{item:tP1link},
we first prove 
the case when $\sigma = o$ as follows.
If $\Delta$ is complete, or more generally, the boundary of a \textit{Gorenstein$^*$} poset,
then the result 
 follows immediately from \cite[Proposition~2.2]{Kar06}.
The result also holds if $\Delta$ is combinatorially equivalent to a complete fan in $\R^d$,
since 
by 
Remark~\ref{remark:cdPureSheavesDependPosetOnly}
combinatorially equivalent fans have equal $t$-Poincar\'e polynomials.

Suppose 
$\Delta$ is non-complete.
From the proof of \cite[Lemma~3.1]{Kar06} we have
\[
P^t(\ov{\cL_\Delta(\Delta)}; t)
=
\Big(
P^t(
\ov{
\cL_{\partial \ov{\Delta}}(\partial \ov{\Delta})
}
; t)
-
P^t(\ov{\cL_\Delta(\partial \Delta)}; t) \cdot 
(1 + t^{T(e_d)})
\Big)
+
P^t(\ov{\cL_\Delta(\partial \Delta)}; t),
\]
where
$\partial \ov{\Delta}$ is the boundary of a \textit{Gorenstein$^*$} poset, 
obtained by adding an element with boundary $\partial \Delta$ to $\Delta$ \cite[Section~3.1]{Kar06}.
Here, 
we use the theory of pure sheaves on a poset \cite[Section~2.4]{Kar06}.
On the other hand,
by \cite[Lemma~3.5]{DKT20},
the $cd$-index of $\Delta$ is given by 
\begin{align*}
\Phi_\Delta(c, d)
& =
\ell^\Phi_\Delta(c, d) 
+ 
\Phi_{\partial \Delta}(c, d)\\
& =
\big(
\Phi_{\partial \ov{\Delta}}(c, d)
-
\Phi_{\partial \Delta}(c, d) \cdot c
\big)
+ 
\Phi_{\partial \Delta}(c, d).
\end{align*}
Hence the result follows from 
applying $\eta$ and using 
the case when 
$\Delta$ is complete.

Now suppose $\sigma > o$.
Since $\lk_\Delta \sigma$ is quasi-convex by Proposition~\ref{p:linkofqcisqc},
we apply the case when $\sigma = o$
to $\lk_\Delta \sigma$ and get
\begin{align*}
\eta(\Phi_{\lk_\Delta \sigma}(c, d) ; t^{2^{\dim(\sigma)}}) 
& = 
P^t(\ov{\cL_{\lk_\Delta \sigma}(\lk_\Delta \sigma) }; t^{2^{\dim(\sigma)}})\\
& = P^t(\ov{\cL_\sigma(\Delta)}; t),
\end{align*}
where the last equality follows Lemma~\ref{lem:SimpleSheafOfLinkIsShiftInDegree}.

Statement \eqref{item:tPlocal} was proved in the proof of \cite[Theorem~2.7]{EK07}:
If $\Delta$ is supported on a cone $\sigma$, 
then $\Delta$ is a refinement of the fan $\langle\sigma\rangle$ and hence there is a natural fan subdivision $\pi\colon \Delta\to \langle\sigma\rangle$.
Recall that 
$K_\sigma = \ker(\ov{\res}
\colon 
\ov{\pi_*\cL_\Delta(\sigma)} \to \ov{\pi_*\cL_\Delta(\partial \sigma) }
)$
is the kernel of the surjective map $\ov{\res}$.
Thus we have
\begin{align*}
L^t(\pi_* \cL_\Delta, \sigma; t)
=
P^t(K_\sigma; t)
& = 
P^t (\ov{\pi_*\cL_\Delta(\sigma)}; t)
-
P^t (\ov{\pi_*\cL_\Delta(\partial \sigma)}; t)\\
& = 
P^t (\ov{\cL_\Delta(\Delta)}; t)
-
P^t (\ov{\cL_\Delta(\partial \Delta)}; t)\\
& = 
\eta(\Phi_\Delta(c, d); t)
-
\eta(\Phi_{\partial \Delta}(c, d); t)\\
& = 
\eta(\ell^\Phi_\Delta(c, d); t),
\end{align*}
where 
the second to the last 
equality
follows from 
Statement~\eqref{item:tP1link}.
\end{proof}

In the remainder of this subsection, 
we rephrase some known results for the $cd$-index in terms of the $t$-Poincar\'e polynomial.
We note that 
Proposition~\ref{p:tPoinPolyDegree}
and
Proposition~\ref{p:tPoinDual}
also hold for fans that are combinatorially equivalent to a quasi-convex fan, 
since by Remark~\ref{remark:cdPureSheavesDependPosetOnly}
combinatorially equivalent fans have equal $t$-Poincar\'e polynomials.

\begin{prop}\label{p:tPoinPolyDegree}
Let $\Delta$ be a quasi-convex fan in $\R^d$.
Then the $t$-Poincar\'e polynomial 
 $P^t(\ov{\cL_\Delta(\Delta)}; t)$ is of degree at most $\tdim(\Delta)$.
\end{prop}

\begin{proof}
By Proposition~\ref{p:FactsAbouttPoin}
we have 
$P^t(\ov{\cL_\Delta(\Delta)};t)
=
{ \eta}(\Phi_{\Delta }(c, d); t)
$.
Hence the result follows from computing the degree of the polynomial ${ \eta}(\Phi_\Delta(c, d); t)$, which is bounded above by 
$2^{0} + \dots + 2^{d-1} = 2^d - 1 =  \tdim(\Delta)$.
\end{proof}

\begin{prop}
\label{p:tPoinDual}
Let $\Delta$ be a quasi-convex fan in $\R^d$.
Then
\[
t^{\tdim(\Delta)} P^t(\ov{\cL_\Delta(\Delta)}; t^{-1})  = P^t(\ov{\cL_\Delta(\Delta, \partial \Delta)}; t).\]
\end{prop}

\begin{proof}
Recall that $\cL_\Delta(\Delta)$ and $\cL_\Delta(\Delta, \partial \Delta)$ are $(\Z_{\geq 0})^d$-graded
free $C_d$-modules dual to each other in the sense of \eqref{eq:cdduality}:
\[
\cL_\Delta(\Delta, \partial \Delta)
 = 
 \Hom_{C_d} \big(
 \cL_\Delta(\Delta), C_d[-e_1 - \dots - e_d]
\big).
\]
Since both $\cL_\Delta(\Delta)$ and $C_d[-e_1 - \dots - e_d]$ are free $C_d$-modules,
we have
\begin{align*}
\ov{\cL_\Delta(\Delta, \partial \Delta)}
& = 
\ov{
\Hom_{C_d} 	
\big(
\cL_\Delta(\Delta), C_d[-e_1 - \dots - e_d]
\big)
}\\
& =
\Hom_{\R} 
\big(
\ov{\cL_\Delta(\Delta)}
, \ov{C_d[-e_1 - \dots - e_d]}
\big).
\end{align*}
%Hence we have
%\[
%\ov{\cL(\Delta, \partial \Delta)}
%=
%\Hom_{\R} 
%\big(
%\ov{\cL_\Delta(\Delta)}
%, \R[-e_1 - \dots - e_d]
%\big).
%\]
The result then follows from
$
P^t(\ov{C_d[-e_1 - \dots
 -e_d]} ; t )
 =
 t^{T(e_1 + \dots + e_d)} 
 =
 t^{\tdim(\Delta)}$. 
\end{proof}

\begin{prop}\label{p:tPBraden}
Let $\Delta$ be a fan in $\R^d$.
For cones $\tau > \sigma$, the following are true:
\begin{enumerate}
\item \label{item:parta}
The module $\cL_\sigma(\tau)$ is generated in $t$-degrees strictly less than $(\tdim(\tau) - \tdim(\sigma))/2$.

\item \label{item:partb}
The module $\cL_\sigma(\tau, \partial \tau) := \ker(\res\colon \cL_\sigma(\tau)\to \cL_\sigma(\partial \tau))$ is generated in $t$-degrees strictly greater than $(\tdim(\tau) - \tdim(\sigma))/2$.
\end{enumerate}
\end{prop}

\begin{proof}
It suffices to show that 
the polynomial
$P^t( \ov{\cL_\sigma(\tau)} ;  t)$
is of degree strictly less than 
$(\tdim(\tau) - \tdim(\sigma))/2$
and 
the lowest degree term in 
$P^t( \ov{\cL_\sigma(\tau, \partial \tau)} ;  t)$
is of degree strictly greater than $(\tdim(\tau) - \tdim(\sigma))/2$.
By Lemma~\ref{lem:SimpleSheafOfLinkIsShiftInDegree} 
we have
\begin{align*}
P^t( \ov{\cL_\sigma(\tau)} ;  t) &=
P^t(\ov{\cL_{\lk_{\Delta} \sigma} (\ov{\tau})}; t^{2^{\dim(\sigma)}}),\\
P^t( \ov{\cL_\sigma(\tau, \partial \tau)} ; t )
&  = P^t(\ov{\cL_{\lk_{\Delta} \sigma} (\ov{\tau}, \partial \ov{\tau})}; t^{2^{\dim(\sigma)}}).
\end{align*}
Since we have
\begin{align*}
\big(
\tdim(\tau) - \tdim(\sigma) 
\big)
/ 2
& = 
\big(
(2^{\dim(\ov{\tau}) + \dim(\sigma)}  - 1) - (2^{\dim(\sigma)} - 1)
\big)
/ 2
\\
& =
2^{\dim(\sigma)} \cdot
\tdim({\ov{\tau}}) / 2,
\end{align*}
the case when $\sigma > o$ follows 
from 
the case when $\sigma = o$ for the fan $\lk_\Delta \sigma$.

We prove the proposition for the case when $\sigma = o$ as follows. 
Let $\tau > o$ be a cone. 
For Statement~\eqref{item:parta}, 
by the definition of a simple sheaf, 
we have 
$P^t(\ov{\cL_\Delta(\tau)}; t) = P^t(\ov{\cL_\Delta(\partial \tau)}; t)$.
Since $\partial \tau$ is combinatorially 
equivalent to a complete fan,
by 
Proposition~\ref{p:tPoinPolyDegree} 
the polynomial 
$P^t(\ov{\cL_\Delta(\partial \tau)}; t)$
is of degree at most $\tdim(\partial \tau) = 
2^{\dim(\tau) - 1} - 1$,
which is strictly less than $\tdim(\tau)/2 = (2^{\dim(\tau) } - 1)/2$.
For Statement~\eqref{item:partb}, 
since $\langle \tau \rangle$ is a quasi-convex fan in $\Span(\tau)$, 
by Proposition~\ref{p:tPoinDual} we have
\[
t^{\tdim(\tau)} P^t(\ov{\cL_\Delta(\tau, \partial \tau)}; t^{-1}) = P^t(\ov{\cL_\Delta(\tau)}; t),
\]
which by Statement~\eqref{item:parta} is a polynomial of degree strictly less than $\tdim(\tau) /2$.
Hence the lowest degree term in $P^t(\ov{\cL_\Delta(\tau, \partial \tau)}; t)$
is of degree strictly greater than $\tdim(\tau) /2$.
\end{proof}

\begin{remark}
For a quasi-convex fan $\Delta$,
one can prove 
the following recursion relation:
\[
t^{\tdim(\Delta)}
P^t(\ov{\cL_\Delta(\Delta)}; t^{-1})  =
\sum_{\sigma\in \Delta} P^t(\ov{\cL_\Delta(\sigma)}; t)
\prod_{i=\dim(\sigma) + 1}^{d}
(t^{T(e_i)}-1);
\]
see also \cite[p.~709]{Kar06}.
This recursion relation is analogous to the definition of the toric $h$-polynomial,
and together with 
 Proposition~\ref{p:tPBraden} 
characterizes 
$P^t(\ov{\cL_\Delta(\Delta)}; t)$ for any quasi-convex fan $\Delta$. 
\end{remark}

\subsection{Weight sheaves for the \texorpdfstring{$cd$}{cd}-index}

For a pure sheaf $\cF$ on a fan $\Delta$ in $\R^d$
and an integer $r\in \Z$, 
we define $\cW^r \cF$, 
a subsheaf of $(\Z_{\geq 0})^d$-graded $\cC$-modules of $\cF$, called the \emph{$r$-weight sheaf} of $\cF$ as follows:
for each cone $\sigma$, we set
\[
\cW^r \cF (\sigma) = \begin{cases}
\cF(\sigma) & \text{if } r \leq 0, \\
\res^{-1}(\cW^r \cF(\partial \sigma)) & \text{if } 0<r \leq \tdim(\sigma),\\
\sum_{i=1}^d x_i \cW^{r-2 \cdot \tdeg(x_i)}\cF(\sigma)
 & \text{if } r > \tdim(\sigma).
\end{cases}
\]
In particular,
on an open set $U$, we have
\[
\cW^r\cF(U) = \{ (f_\sigma)_{\sigma \in U} :
f_\sigma\in \cW^r\cF(\sigma) ,
\res_{\sigma, \tau}( f_{\sigma}) = f_\tau \text{ for all } \tau \leq \sigma
\}.\]
We prove that the weight sheaves are indeed well-defined subsheaves of $\cF$ in Proposition~\ref{p:cdWeightSheavesAreSubsheaves}.
For now, 
we consider each weight sheaf
an assignment to each open set $U$
of a $\cC(U)$-submodule of $\cF(U)$.

\begin{prop}\label{p:cdMultiplyxiIncreaseTwo}
For every open set $U$, integer $r$ and $1\leq i\leq d$, we have
\[
x_i \cW^{r-2 \cdot \tdeg(x_i)}\cF(U) \subseteq \cW^{r}\cF(U).
\]
\end{prop}

\begin{proof}
The proof is similar to that of Proposition~\ref{p:madds2}:
We want to prove that for each $1\leq i \leq d$,
\[
(f_\sigma)_{\sigma\in U} \in \cW^{r-2\cdot \tdeg(x_i)}\cF(U) \implies
(x_i  f_\sigma)_{\sigma\in U} \in \cW^{r}\cF(U).
\]
It suffices to prove the statement for $U = \langle \sigma \rangle$ whenever $\sigma$ is a cone, which we do by induction on $\dim(\sigma)$.
The statement is clear for the zero cone $o$.
By induction we may assume the statement for $U = \partial \sigma$.

The case when $r\leq 0$ or $r> \tdim(\sigma)$ is clear.
Suppose $0 < r \leq \tdim(\sigma)$.
Then for each $1\leq i \leq d$
we have
\[
\res \big(x_i \cW^{r-2\cdot \tdeg(x_i)}\cF(\sigma) \big)
=
\begin{cases}
x_i \res(\cF(\sigma))
& \text{if } 0 < r \leq 2 \cdot \tdeg(x_i) \\
x_i \res(\res^{-1}(\cW^{r-2\cdot \tdeg(x_i)}\cF(\partial \sigma)))
& \text{otherwise.}
\end{cases}
\]
Since $\res$ is surjective,
we have $\res(\cF(\sigma)) = \cF(\partial \sigma) = \cW^{r-2\cdot \tdeg(x_i)} \cF(\partial \sigma)$ when $0 < r \leq 2 \cdot \tdeg(x_i)$.
Hence for $0 < r \leq \tdim(\sigma)$ we have
\[
\res(x_i \cW^{r-2\cdot \tdeg(x_i)}\cF(\sigma))
=
x_i \cW^{r-2\cdot \tdeg(x_i)}\cF(\partial \sigma).
\]
By the induction hypothesis, we also have $x_i \cW^{r-2\cdot \tdeg(x_i)} \cF(\partial \sigma) \subseteq \cW^{r}\cF(\partial \sigma)$.
Now, 
by taking the preimage,
we have 
\begin{align*}
x_i \cW^{r-2\cdot \tdeg(x_i)}\cF(\sigma)
& \subseteq 
\res^{-1}
\big(
\res
(
x_i \cW^{r-2\cdot \tdeg(x_i)}\cF(\sigma)
)
\big) \\
& \subseteq 
\res^{-1}
\big(
x_i 
\cW^{r-2\cdot \tdeg(x_i)} \cF(\partial \sigma)
\big) \\
& \subseteq 
\res^{-1}
\big(
\cW^r \cF(\partial \sigma)
\big) \\
& =
\cW^r \cF(\sigma).
\qedhere
\end{align*}
\end{proof}

We check the well-definedness of the weight sheaves.

\begin{prop}\label{p:cdWeightSheavesAreSubsheaves}
The assignment $\cW^r \cF$ together with the restriction maps of $\cF$ form a sheaf of $(\Z_{\geq 0})^d$-graded $\cC$-modules.
\end{prop}

\begin{proof}
The proof is similar to that of Proposition~\ref{p:weightissheaf}:
It remains to show that the restriction map $\res_\cF\colon \cF(\sigma) \to \cF(\partial \sigma)$ satisfies
\[
\res_\cF(\cW^r\cF(\sigma)) \subseteq \cW^r\cF(\partial \sigma)
\]
for every cone $\sigma$ and every integer $r$,
which we prove by induction on $r$.
This is clear if $r\leq \tdim(\sigma)$.
For $r>\tdim (\sigma)$, we have
\begin{align*}
\res_{\cF}
\big(
\sum_{i=1}^d x_i \cW^{r-2 \cdot \tdeg(x_i)}\cF(\sigma)
\big) 
& =
\sum_{i=1}^d x_i
\res_{\cF}
\big(
\cW^{r-2 \cdot \tdeg(x_i)}\cF(\sigma)
\big)
\\
& \subseteq 
\sum_{i=1}^d 
x_i
\cW^{r-2 \cdot \tdeg(x_i)}\cF(\partial \sigma)\\
& \subseteq\cW^r\cF(\partial \sigma),
\end{align*}
where the first inclusion follows from the induction hypothesis and the second inclusion follows from Proposition~\ref{p:cdMultiplyxiIncreaseTwo}.
\end{proof}

We prove a characterization of the weight sheaves of a shifted simple sheaf, analogous to Proposition~\ref{p:charofweight}, as follows.
For a $(\Z_{\geq 0})^d$-graded vector space $M$, 
let $M^{\tdeg \geq r}$, $M^{\tdeg \leq r}$ and $M^{\tdeg = r}$
be the vector subspaces of $M$ generated by homogeneous elements of $t$-degree at least $r$, at most $r$ and exactly $r$, respectively.

\begin{prop}\label{p:cdcharofweight}
Let $\cF = \cL_\sigma[-{j}]$ be a shifted simple sheaf on a fan $\Delta$, where ${j} \in (\Z_{\geq 0})^d$.
Then for every open set $U$ and every integer $r$ we have
\[
\cW^r \cF(U) =
\cF(U)^{\tdeg \geq T({j})+\frac{r - \tdim(\sigma)}{2}}.
\]
\end{prop}

\begin{proof}
The proof is similar to that of Proposition~\ref{p:charofweight}:
It suffices to prove the statement for each $U = \langle \tau\rangle$ whenever $\tau$ is a cone,
which we do by induction on $\dim(\tau)$.
The case when $\tau = \sigma$ is trivial.

For a cone $\tau > \sigma$, by induction we may assume the statement for $U = \partial \tau$.
We further induct on $r$, 
and prove the base case as follows. 
The case when $r\leq 0$ is clear. 
Suppose $0 < r \leq \tdim(\tau)$.
Since the restriction map is homogeneous of degree $0$
and 
$\cW^r \cF(\partial \tau) = \cF(\partial \tau)^{\tdeg \geq T(j) + \frac{r-\tdim(\sigma)}{2}}$ by the induction hypothesis of the first induction, 
we have
\begin{align*}
\cW^r\cF(\tau) =
\res^{-1}(\cW^r\cF(\partial \tau))
& =
\cF(\tau)^{\tdeg \geq T({j})+\frac{r - \tdim(\sigma)}{2}}
+
\cF(\tau, \partial \tau)\\
& =
\cF(\tau)^{\tdeg \geq T({j})+\frac{r - \tdim(\sigma)}{2}},
\end{align*}
where 
the last equality follows from Proposition~\ref{p:tPBraden}.

Suppose $r> \tdim(\tau)$.
For each $1\leq i \leq d$,
by considering the $t$-degrees, 
we have
\[
x_i \cF(\tau)^{\tdeg \geq T({j})+\frac{r - \tdim(\sigma)}{2} - \tdeg(x_i)}
\subseteq
\cF(\tau)^{\tdeg \geq T({j})+\frac{r - \tdim(\sigma)}{2}}.
\]
Thus we have
\[
\sum_{i=1}^d x_i \cF(\tau)^{\tdeg \geq T({j})+\frac{(r -2\cdot \tdeg(x_i) )- \tdim(\sigma)}{2}}
\subseteq
\cF(\tau)^{\tdeg \geq T({j})+\frac{r - \tdim(\sigma)}{2}}.
\]
The other inclusion also holds:
By Proposition~\ref{p:tPBraden},
the module $\cF(\tau)$ is generated in $t$-degrees strictly less than $T({j}) + \frac{\tdim(\tau) - \tdim(\sigma)}{2}$.
In other words, a homogeneous element of $t$-degree at least 
$T({j}) + \frac{\tdim(\tau) - \tdim(\sigma)}{2}$
can be generated using elements of lower $t$-degrees.
Since $r > \tdim(\tau)$,
we have 
$T({j})+\frac{r - \tdim(\sigma)}{2} > T({j}) + \frac{\tdim(\tau) - \tdim(\sigma)}{2}$,
hence the other inclusion also holds.
Now, by the induction hypothesis of the second induction, 
for each $1\leq i \leq d$ we have 
\[
\cW^{r -2\cdot \tdeg(x_i)}\cF(\tau) = \cF(\tau)^{\tdeg \geq T({j})+\frac{(r -2\cdot \tdeg(x_i) )- \tdim(\sigma)}{2}}.
\]
Hence we have
\begin{align*}
\cW^r\cF(\tau)  = \sum_{i=1}^d x_i \cW^{r -2\cdot \tdeg(x_i)}\cF(\tau)
& = \sum_{i=1}^d x_i 
\cF(\tau)^{\tdeg \geq T({j})+\frac{(r -2\cdot \tdeg(x_i) )- \tdim(\sigma)}{2}}\\
& = \cF(\tau)^{\tdeg \geq T({j})+\frac{r - \tdim(\sigma)}{2}}. \qedhere
\end{align*}
\end{proof}

\begin{prop}\label{p:cdWeightSheavesAreDecreasing}
Let $\cF$ be a pure sheaf on a fan $\Delta$.
Then the weight sheaves form a decreasing sequence of subsheaves of $\cF$, i.e.,
for each open set $U$ and $r\in  \Z$ 
we have
$\cW^r\cF(U) \supseteq \cW^{r+1}\cF(U)$.
\end{prop}

\begin{proof}
The proof is similar to that of Proposition~\ref{p:WeightSheavesAreDecreasing}:
We fix a decomposition of $\cF$ into shifts of simple sheaves:
\[
\cF \simeq \oplus_\alpha \cL_{\sigma_\alpha}[-j_\alpha],
\]
where each $j_\alpha\in (\Z_{\geq 0})^d$.
Since the construction of the weight sheaves commutes with direct sums, 
for each $r \in \Z$ we have
\[
\cW^r\cF \simeq \oplus_\alpha 
\cW^r
\cL_{\sigma_\alpha}[-j_\alpha].
\]
Hence it suffices to prove the statement for a shifted simple sheaf $\cL_{\sigma_\alpha}[-j_\alpha]$.
By Proposition~\ref{p:cdcharofweight},
we have 
$\cW^r
\cL_{\sigma_\alpha}[-j_\alpha](U) = \cL_{\sigma_\alpha}[-j_\alpha](U)^{\tdeg \geq T(j_\alpha)
+ \frac{r-\tdim(\sigma_\alpha)}{2}}$, 
which is naturally a decreasing filtration.
\end{proof}

For a pure sheaf $\cF$ on $\Delta$, we define filtrations on the $(\Z_{\geq 0})^d$-graded vector space $\ov{\cF(\Delta)}$ as follows.
The \emph{Hodge filtration} $F_\bullet$ of $\cF$ is defined to be the increasing filtration on $\ov{\cF(\Delta)}$ given by
\[
F_p := \ov{\cF(\Delta)}
^{\tdeg \leq p}.
\]
Thus we also have
\[
P^t(\ov{\cF(\Delta)}; t)  = \sum_p
\dim_\R (\Gr^F_p
\ov{\cF(\Delta)}
)
t^p.
\]
The \emph{weight filtration} $W^\bullet$ of $\cF$ is defined to be the decreasing filtration on $\ov{\cF(\Delta)}$ given by
\[
W^r := \rho(\cW^r \cF(\Delta)),
\]
where $\rho\colon \cF(\Delta)\to \ov{\cF(\Delta)}$ is the quotient map.
This is a well-defined filtration, since by Proposition~\ref{p:cdWeightSheavesAreDecreasing} the weight sheaves are a decreasing sequence of subsheaves.
Both filtrations commute with direct sums:
for every decomposition $\cF \simeq \oplus_\alpha \cL_{\sigma_\alpha}[-j_\alpha]$ of $\cF$ into shifts of simple sheaves, the Hodge (resp. weight) filtration of $\cF$ is the direct sum of the Hodge (resp. weight) filtrations of the shifted simple sheaves.

The graded vector space $\ov{\cF(\Delta)}$ is now equipped with the weight filtration $W^\bullet$ and the Hodge filtration $F_\bullet$.
As defined in Definition~\ref{defn:HodgeDelignePoly}, 
its \emph{Hodge--Deligne polynomial} 
is 
\[
E(\ov{\cF(\Delta)}, F_\bullet, W^\bullet; u, v) =
\sum_{p, q} \dim_\R
(
\Gr^F_p \Gr_W^{p+q} \ov{\cF(\Delta)}
)
u^p v^q,
\]
which 
specializes to its $t$-Poincar\'e polynomial under $u\mapsto t$ and $v\mapsto 1$.

\subsection{The mixed \texorpdfstring{$cd$}{cd}-index and the Hodge--Deligne polynomial}

We first define a linear map ${ \eta}'\colon \R\langle c', d'\rangle \otimes_\R \R\langle c, d \rangle  \to \R[u, v]$ as follows:
given homogeneous polynomials $f(c', d')$ and $g(c, d)$, 
we set
\begin{align*}\label{eq:cdchangeofvariables}
\eta'\big(f(c', d')\otimes g(c, d)\big)
=
v^{2^{\deg(f)}-1}
\eta \big(f(c, d); u v^{-1} \big)
\cdot
\eta \big(g(c, d); (u v)^{2^{\deg(f)}} \big).
\end{align*}
The map is well-defined, 
since $\eta(f(c, d); t)$ is a polynomial of degree at most $2^{\deg(f)} - 1$.
As a result, 
the image of the mixed $cd$-index $\Omega_\pi$ under $\eta'$ is given by
\begin{equation}\label{eq:goodcddefnuv}
\eta'
\big(
\Omega_\pi
(c', d', c, d)
\big) 
=
\sum_{\sigma\in \Delta}
v^{\tdim(\sigma)}
\eta
\big(
\ell^\Phi_{\pi^{-1}(\langle \sigma \rangle)} (c, d)
; u v^{-1}
\big)
\cdot
\eta
\big(
\Phi_{\lk_\Delta \sigma}(c, d)
; (u v)^{2^{\dim(\sigma)}}
\big).
\end{equation}
We note that this formula is analogous to Definition~\ref{defn:mixedh}, the definition of the mixed $h$-polynomial.

\begin{prop}
The linear map $\eta'$ is injective.
\end{prop}

\begin{proof}
We first define a linear map $\delta \colon \R[u, v]\to \R[u, v]\otimes_\R \R[u, v]$ as follows:
Given a monomial $m = u^a v^b$ in $R[u, v]$,
there exist unique subsets
$U, V \subseteq  \Z_{\geq 1}$
such that
$
m = \prod_{i\in U} u^{T(e_i)} \cdot 
\prod_{i\in V} v^{T(e_i)}.
$
Let 
\[
m_1  = \prod_{i\in U \setminus V} u^{T(e_i)} \cdot 
\prod_{i\in V \setminus U} v^{T(e_i)}
\ \text{ and } \
m_2 = \prod_{i\in U \cap V} (uv )^{T(e_i)}.
\]
Then we let $\delta$ to be the linear map 
defined by 
$\delta(m) = m_1\otimes m_2$.

Given homogeneous polynomials $f(c', d')$ and $g(c, d)$,
we observe that 
\[
\delta
\Big(
\eta'
\big(
f(c', d')\otimes g(c, d)
\big)
\Big)
=
v^{2^{\deg(f)}-1}
\eta \big(f(c, d); u v^{-1} \big)
\otimes 
\eta \big(g(c, d); (u v)^{2^{\deg(f)}} \big).
\]
Using the fact that 
the first component of the image
is homogeneous of degree $2^{\deg(f)} - 1$,
we can recover $\deg(f)$ and 
hence
$\eta (g(c, d); uv)$.
We can also recover $\eta (g(c, d); u)$ from the first component of the image.
Since $\eta$ is injective, 
we can recover both $f(c', d')$ and $g(c, d)$.
Hence $\eta'$ is injective. 
\end{proof}

Since $\eta'$ is injective, 
the content of the following theorem is that 
the mixed $cd$-index can be uniquely recovered from the Hodge--Deligne polynomial.

\resmixedcd

\begin{proof}
We compute the Hodge--Deligne polynomial of $(\ov{\pi_*\cL_\Sigma(\Delta)}, F_\bullet, W^\bullet)$
following the proof of Lemma~\ref{lem:hodgemixedh}:
We choose a decomposition $\pi_* \cL_\Sigma = \oplus_{\alpha} \cL_{\sigma_{\alpha}}[-{j_\alpha}]$ of $\pi_* \cL_\Sigma$ into shifts of simple sheaves on $\Delta$, 
where each $\sigma_\alpha \in \Delta$ and each $j_\alpha\in (\Z_{\geq 0})^d$.
We fix $\alpha$ and consider the shifted simple sheaf $\cL_{\sigma_\alpha}[-j_\alpha]$.
By Proposition~\ref{p:cdcharofweight}, 
the weight filtration $W^\bullet$ on 
$\ov{\cL_{\sigma_\alpha}[-j_\alpha](\Delta)}$
is given by 
\[
W^r = \ov{\cL_{\sigma_\alpha}[-j_\alpha](\Delta)}^{\tdeg \geq T(j_\alpha) + \frac{r-\tdim(\sigma_\alpha)}{2}}.
\]
Hence the associated graded vector space 
$\Gr_W^r \ov{\cL_{\sigma_\alpha}[-j_\alpha](\Delta)}$
is given by 
\[
\Gr_W^{r} \ov{\cL_{\sigma_\alpha}[-{j_\alpha}](\Delta)} = \begin{cases}
\ov{\cL_{\sigma_\alpha}[-{j_\alpha}](\Delta)}^{ \tdeg = T({j_\alpha})+\frac{r -  \tdim(\sigma_\alpha)}{2}} & \text{if } r -  \tdim(\sigma_\alpha) \text{ is even} ,\\
0 & \text{otherwise.}
\end{cases}
\]

We restrict to the case when $r -  \tdim(\sigma_\alpha)$ is even.
Let 
$k = \frac{r -  \tdim(\sigma_\alpha)}{2}$.
Then 
the associated graded vector space is given by
\[
\Gr^F_p \Gr_W^r \ov{\cL_{\sigma_\alpha}[-{j_\alpha}](\Delta)}
=
\begin{cases}
\ov{\cL_{\sigma_\alpha}[-{j_\alpha}](\Delta)}^{ \tdeg = T({j_\alpha})+k}
& \text{if } p = T({j_\alpha})+k,
 \\
0 & \text{otherwise.}
\end{cases}
\]
Thus the Hodge--Deligne polynomial of $(\ov{\pi_*\cL_\Sigma(\Delta)}, F_\bullet, W^\bullet)$ is given by
\begin{align*}
& E(\ov{\pi_* \cL_\Sigma(\Delta)}, F_\bullet, W^\bullet; u, v) \\
 & \qquad=
\sum_{\alpha}
\sum_{p, r}
\dim_\R ( \Gr^F_p \Gr_W^r \ov{\cL_{\sigma_\alpha}[-j_\alpha](\Delta)})
u^p v^{r-p}\\
&\qquad =
\sum_{\alpha}
\sum_{k}
\dim_\R ( \ov{\cL_{\sigma_\alpha}[-j_\alpha](\Delta)}^{\tdeg = T({j_\alpha})+k}
)
u^{T(j_\alpha)+k} v^{\tdim(\sigma_\alpha)+k-T(j_\alpha)},
\end{align*}
%since the associated graded vector space is zero unless 
%$p = T(j_\alpha) + k$ and 
%$r = 2k + \tdim(\sigma_\alpha)$ for some integer $k$.
By removing the shift, 
we have
%\[\dim_\R ( \ov{\cL_{\sigma_\alpha}[-j_\alpha](\Delta)}^{\tdeg = T({j_\alpha})+k}) = \dim_\R ( \ov{\cL_{\sigma_\alpha}(\Delta)}^{\tdeg = k}).\]
%Thus we rearrange the sum and get
\begin{align*}
E(\ov{\pi_* \cL_\Sigma(\Delta)}, F_\bullet, W^\bullet; u, v) 
& =
\sum_{\alpha}
\sum_{k}
\dim_\R ( \ov{\cL_{\sigma_\alpha}(\Delta)}^{\tdeg = k}
)
u^{T(j_\alpha)+k} v^{\tdim(\sigma_\alpha)+k-T(j_\alpha)}
\\
&=
\sum_{\sigma\in \Delta}
v^{\tdim(\sigma)}
\sum_{\substack{\alpha\\ \sigma_\alpha = \sigma} }
(u v^{-1})^{T({j_\alpha})}
\sum_{k}
\dim_\R ( \ov{\cL_{\sigma}(\Delta)}^{\tdeg = k})
(uv)^{k}\\
&=
\sum_{\sigma\in \Delta}
v^{\tdim(\sigma)}
L^t(\pi_*\cL_\Sigma, \sigma; u v^{-1})
\cdot
P^t(\ov{\cL_\sigma(\Delta)}; uv).
\end{align*}
%where the last equality follows from the definitions of the $t$-Poincar\'e polynomial and the local $t$-Poincar\'e polynomial.
By Proposition~\ref{p:FactsAbouttPoin},
we have
\begin{align*}
P^t(\ov{\cL_\sigma(\Delta)}; uv) & = \eta(\Phi_{\lk_\Delta \sigma}(c, d); (uv)^{2^{\dim(\sigma)}}), \\
L^t(\pi_*\cL_\Sigma, \sigma; u v^{-1}) & =
\eta(\ell^\Phi_{\pi^{-1}(\langle \sigma \rangle)}(c, d); u v^{-1}).
\end{align*}
Hence the Hodge--Deligne polynomial is equal to 
\[
E(\ov{\pi_* \cL_\Sigma(\Delta)}, F_\bullet, W^\bullet; u, v) 
=
\sum_{\sigma\in \Delta}
v^{\tdim(\sigma)}
\eta(\ell^\Phi_{\pi^{-1}(\langle \sigma \rangle)}(c, d); u v^{-1})
\cdot
\eta(\Phi_{\lk_\Delta \sigma}(c, d); (uv)^{2^{\dim(\sigma)}}),
\]
which is precisely 
\eqref{eq:goodcddefnuv}, the image of the mixed $cd$-index under $\eta'$.
\end{proof}

\begin{remark}\label{rem:holdforGorenstein}
We expect the theorem to hold for a Gorenstein$^*$ subdivision of a near-Gorenstein$^*$ poset;
see \cite[Section~2]{EK07} for their definitions.
In fact, 
it suffices to show that 
if 
$\Lambda$ is the boundary of a Gorenstein$^*$ poset or is a near-Gorenstein$^*$ poset, then 
$\cL_\Lambda(\Lambda)$ and $\cL_\Lambda(\Lambda, \partial \Lambda)$ are free $C_{\rank(\Lambda)}$-modules dual to each other in the sense of \eqref{eq:cdduality}.
\end{remark}

\end{document}